\documentclass[11pt]{article}
\usepackage{amsmath,amssymb,amsthm,amscd}

\newtheorem{theorem}{Theorem}[section]
\newtheorem{lemma}[theorem]{Lemma}
\newtheorem{proposition}[theorem]{Proposition}

\theoremstyle{plain}

\theoremstyle{definition}
\newtheorem{definition}[theorem]{Definition}
\newtheorem{remark}[theorem]{Remark}

\numberwithin{equation}{section}

\renewcommand{\labelenumi}{\textup{(\theenumi)}}

\newcommand{\End}{\operatorname{End}}
\newcommand{\Homeo}{\operatorname{Homeo}}

\newcommand{\id}{\operatorname{id}}
\newcommand{\Ker}{\operatorname{Ker}}

\newcommand{\Ad}{\operatorname{Ad}}
\newcommand{\Tr}{\operatorname{Tr}}

\newcommand{\K}{\mathcal{K}}
\newcommand{\C}{\mathcal{C}}

\newcommand{\N}{\mathbb{N}}
\newcommand{\Z}{\mathbb{Z}}
\newcommand{\T}{\mathbb{T}}
\newcommand{\Zp}{{\mathbb{Z}}_+}

\def\R{{\mathcal{R}}}
\def\WRA{{\widetilde{\R}_A}}

\title{Topological conjugacy of topological Markov shifts and Ruelle algebras }
\author{Kengo Matsumoto \\
Department of Mathematics \\
Joetsu University of Education \\
Joetsu, 943-8512, Japan
}

\begin{document}
\maketitle

\date{ }

\def\det{{{\operatorname{det}}}}

\begin{abstract}
We will characterize topologically conjugate two-sided topological Markov shifts
$(\bar{X}_A,\bar{\sigma}_A)$
in terms of the associated asymptotic Ruelle $C^*$-algebras 
${\mathcal{R}}_A$
with its commutative $C^*$-subalgebras $C(\bar{X}_A)$
 and the canonical circle actions.  
We will also show that  extended Ruelle algebras $\WRA$, 
which are unital, purely infinite version of the asymptotic Ruelle algebras, 
with its commutative $C^*$-subalgebras $C(\bar{X}_A)$
and the canonical torus actions $\gamma^A$ are complete invariants for topological conjugacy of  two-sided topological Markov shifts.  
We then have a computable topological conjugacy invariant, written in terms of the underlying matrix, of a two-sided topological Markov shift by using K-theory of the extended Ruelle algebra. 
The diagonal action of $\gamma^A$ has a unique KMS-state on 
$\WRA$, which is an extension of the Parry measure on $\bar{X}_A$.
\end{abstract}

\def\R{{\mathcal{R}}}
\def\OA{{{\mathcal{O}}_A}}
\def\OB{{{\mathcal{O}}_B}}
\def\RA{{{\mathcal{R}}_A}}
\def\WRA{{\widetilde{\R}_A}}
\def\WRB{{\widetilde{\R}_B}}
\def\RAC{{{\mathcal{R}}_A^\circ}}
\def\RAR{{{\mathcal{R}}_A^\rho}}
\def\RTA{{{\mathcal{R}}_{A^t}}}
\def\RB{{{\mathcal{R}}_B}}
\def\FA{{{\mathcal{F}}_A}}
\def\FB{{{\mathcal{F}}_B}}
\def\FA{{{\mathcal{F}}_A}}
\def\FTA{{{\mathcal{F}}_{A^t}}}
\def\FB{{{\mathcal{F}}_B}}
\def\DTA{{{\mathcal{D}}_{{}^t\!A}}}
\def\FDA{{{\frak{D}}_A}}
\def\FDTA{{{\frak{D}}_{{}^t\!A}}}
\def\FFA{{{\frak{F}}_A}}
\def\FFTA{{{\frak{F}}_{A^t}}}
\def\FFB{{{\frak{F}}_B}}
\def\OZ{{{\mathcal{O}}_Z}}
\def\V{{\mathcal{V}}}
\def\E{{\mathcal{E}}}
\def\OTA{{{\mathcal{O}}_{A^t}}}
\def\OTB{{{\mathcal{O}}_{B^t}}}
\def\SOA{{{\mathcal{O}}_A}\otimes{\mathcal{K}}}
\def\SOB{{{\mathcal{O}}_B}\otimes{\mathcal{K}}}
\def\SOZ{{{\mathcal{O}}_Z}\otimes{\mathcal{K}}}
\def\SOTA{{{\mathcal{O}}_{A^t}\otimes{\mathcal{K}}}}
\def\SOTB{{{\mathcal{O}}_{B^t}\otimes{\mathcal{K}}}}
\def\DA{{{\mathcal{D}}_A}}
\def\DB{{{\mathcal{D}}_B}}
\def\DZ{{{\mathcal{D}}_Z}}

\def\SDA{{{\mathcal{D}}_A}\otimes{\mathcal{C}}}
\def\SDB{{{\mathcal{D}}_B}\otimes{\mathcal{C}}}
\def\SDZ{{{\mathcal{D}}_Z}\otimes{\mathcal{C}}}
\def\SDTA{{{\mathcal{D}}_{A^t}\otimes{\mathcal{C}}}}
\def\O2{{{\mathcal{O}}_2}}
\def\D2{{{\mathcal{D}}_2}}

\def\Max{{{\operatorname{Max}}}}
\def\Tor{{{\operatorname{Tor}}}}
\def\Ext{{{\operatorname{Ext}}}}
\def\Per{{{\operatorname{Per}}}}
\def\PerB{{{\operatorname{PerB}}}}
\def\Homeo{{{\operatorname{Homeo}}}}
\def\HA{{{\frak H}_A}}
\def\HB{{{\frak H}_B}}
\def\HSA{{H_{\sigma_A}(X_A)}}
\def\Out{{{\operatorname{Out}}}}
\def\Aut{{{\operatorname{Aut}}}}
\def\Ad{{{\operatorname{Ad}}}}
\def\Inn{{{\operatorname{Inn}}}}
\def\det{{{\operatorname{det}}}}
\def\exp{{{\operatorname{exp}}}}
\def\cobdy{{{\operatorname{cobdy}}}}
\def\Ker{{{\operatorname{Ker}}}}
\def\ind{{{\operatorname{ind}}}}
\def\id{{{\operatorname{id}}}}
\def\supp{{{\operatorname{supp}}}}
\def\co{{{\operatorname{co}}}}
\def\Sco{{{\operatorname{Sco}}}}
\def\U{{{\mathcal{U}}}}

\section{Introduction}

A Smale space $(X,\phi)$ is a hyperbolic dynamical system having a local product structure (cf. \cite{Bowen1}, \cite{Smale}).
A two-sided topological Markov shift $(\bar{X}_A,\bar{\sigma}_A)$
gives a typical example of Smale space.  
D. Ruelle in \cite{Ruelle1}, \cite{Ruelle2} introduced $C^*$-algebras from a Smale space $(X,\phi)$.
After the Ruelle's work, 
I. Putnam in \cite{Putnam1}, \cite{Putnam2}
 has initiated to study structure of these $C^*$-algebras by using groupoid technique
(for further studies, see \cite{KamPutSpiel}, \cite{Putnam3}, \cite{Putnam4}, \cite{PutSp}, 
\cite{Thomsen}, etc. ).
For a Smale space $(X,\phi)$, Putnam considered three kinds of $C^*$-algebras
$S(X,\phi), U(X,\phi)$ and $A(X,\phi)$ 
and their crossed products 
$S(X,\phi)\rtimes\Z, U(X,\phi)\rtimes\Z$ and $A(X,\phi)\rtimes\Z$ 
induced by the original homeomorphisms $\phi$, respectively. 
The algebras $S(X,\phi), U(X,\phi)$ and $A(X,\phi)$
are the $C^*$-algebras of the groupoids of stable equivalence relation on $X$, 
 unstable equivalence relation on $X$
 and
 asymptotic  equivalence relation on $X$,
respectively.
I. Putnam has pointed out that 
if the Smale space $(X,\phi)$ is a two-sided topological Markov shift
$(\bar{X}_A,\bar{\sigma}_A)$ defined by an irreducible matrix $A$,
the $C^*$-algebras
$S(X,\phi), U(X,\phi)$ are isomorphic to AF-algebras 
$\FA\otimes\K, \mathcal{F}_{A^t}\otimes\K$,
and 
$S(X,\phi)\rtimes\Z, U(X,\phi)\rtimes\Z$
are isomorphic to
$\SOA, \mathcal{O}_{A^t}\otimes\K$
where $\OA, \FA$ are the Cuntz--Krieger algebra,
the canonical AF-subalgebra of $\OA$, respectively for the matrix $A$,
and
$\mathcal{F}_{A^t}, \mathcal{O}_{A^t}$
are those ones for the transposed matrix $A^t$ of $A$,
and
$\K$ is the $C^*$-algebra of compact operators on separable infinite dimensional Hilbert space
$\ell^2(\N)$.

In \cite{MaPre2017}, the author has introduced  notions of asymptotic continuous orbit equivalence
and asymptotic conjugacy in Smale spaces,
and studied relationship with the crossed product $A(X,\phi)\rtimes\Z$
of the asymptotic Ruelle $C^*$-algebra.  
In this paper, we will restrict  our interest in Smale spaces to two-sided topological Markov shifts.
Let $A$ be an $N\times N$ irreducible non-permutation matrix with entries in $\{0,1\}$.
The shift space 
$\bar{X}_A$ of the two-sided topological Markov shift
$(\bar{X}_A,\bar{\sigma}_A)$ is defined by the compact metric space of bi-infinite sequences 
$(x_i)_{i\in\Z}$ satisfying $A(x_i,x_{i+1}) =1, i \in \Z$ 
with shift transformation
$\bar{\sigma}_A((x_i)_{i\in \Z}) =(x_{i+1})_{i\in \Z}$,
where metric $d$ on $\bar{X}_A$ is defined by
\begin{equation*}
d((x_n)_{n\in \Z}, (y_n)_{n\in \Z}) =
\begin{cases}
0 & \text{ if } (x_n)_{n\in \Z} = (y_n)_{n\in \Z}, \\
1 & \text{ if } x_0 \ne y_0,\\
(\lambda_0)^{k+1} & \text{ if } k = \Max\{|n| \mid x_i = y_i \text{ for all } i 
\text{ with } |i| \le n\}
\end{cases}
\end{equation*}
 for some fixed real number $0<\lambda_0<1$. 
Let $G_A^a$ be the asymptotic \'etale groupoid for $(\bar{X}_A,\bar{\sigma}_A)$
defined by the asymptotic equivalence relation
\begin{equation*}
G_A^a =\{(x, z) \in \bar{X}_A\times\bar{X}_A 
\mid
\lim_{n\to\infty}d(\bar{\sigma}_A^n(x), \bar{\sigma}_A^n(z) )
 =\lim_{n\to\infty}d(\bar{\sigma}_A^{-n}(x), \bar{\sigma}_A^{-n}(z) )
=0\}. 
\end{equation*}
There are natural groupoid operations on $G_A^a$ with topology which makes 
the groupoid $G_A^a$ \'etale (see \cite{Putnam1}, \cite{Putnam2}).
For a general theory for \'etale groupoids, see \cite{AnanRenault}, \cite{Renault1}, 
\cite{Renault2}, \cite{Renault3}, etc.   
As in \cite{Putnam2}, the $C^*$-algebra  $A(\bar{X}_{\bar{\sigma}_A},\bar{\sigma}_A)\rtimes\Z$
is realized as the $C^*$-algebra $C^*(G_A^a\rtimes\Z)$ of the \'etale groupoid 
\begin{equation*}
G_A^a\rtimes\Z
 = \{(x,n,z)  \in \bar{X}_A\times\Z\times\bar{X}_A \mid
 (\bar{\sigma}_A^k(x), \bar{\sigma}_A^l(z)) \in G_A^a, n =k-l\}.
\end{equation*}
The $C^*$-algebra is denoted by $\RA$
and called the asymptotic Ruelle algebra in this paper.
Let $d_A:G_A^a\rtimes\Z\longrightarrow\Z$ be the groupoid homomorphism
defined by
$d_A(x,n,z) =n$.
As the unit space $(G_A^a\rtimes\Z)^\circ$ of $G_A^a\rtimes\Z$ is homeomorphic to $\bar{X}_A$,
the commutative $C^*$-algebra $C(\bar{X}_A)$ is naturally regarded as a subalgebra of $\RA$.
As the algebra 
$\RA$ is a crossed product $A(\bar{X}_{\bar{\sigma}_A},\bar{\sigma}_A)\rtimes\Z$,
it has the dual action $\rho^A_t$ of $t \in \T$. 
In \cite{MaPre2017}, 
an extended version $G_A^{s,u}\rtimes\Z^2$ of the groupoid
$G_A^a\rtimes\Z$ is introduced by setting 
\begin{equation*}
G_A^{s,u}\rtimes\Z^2
 = \{(x,p,q,z) \in \bar{X}_A\times\Z\times\Z\times\bar{X}_A
\mid (\bar{\sigma}_A^p(x), z)\in G_A^s, (\bar{\sigma}_A^q(x),z)\in G_A^u \}
\end{equation*}
where
\begin{align*}
G_A^s =& \{
(x, z) \in \bar{X}_A\times\bar{X}_A 
\mid
\lim_{n\to\infty}d(\bar{\sigma}_A^n(x), \bar{\sigma}_A^n(z) ) =0 \}, \\ 
G_A^u =& \{
(x, z) \in \bar{X}_A\times\bar{X}_A 
\mid
\lim_{n\to\infty}d(\bar{\sigma}_A^{-n}(x), \bar{\sigma}_A^{-n}(z) ) =0 \}.
\end{align*}
There are natural groupoid operations on $G_A^{s,u}\rtimes\Z^2$ 
with topology which makes 
$G_A^{s,u}\rtimes\Z^2$ \'etale (see \cite{MaPre2017}). 
Let $c_A:G_A^{s,u}\rtimes\Z^2\longrightarrow\Z^2$ be the groupoid homomorphism
defined by
$c_A(x,p,q,z) =(p,q)$.
The groupoid $C^*$-algebra $C^*(G_A^{s,u}\rtimes\Z^2)$ is denoted by
$\WRA$ which was denoted by $\R^{s,u}_A$ in \cite{MaPre2017}.
Since the unit space $(G_A^{s,u}\rtimes\Z^2)^\circ$ of $G_A^{s,u}\rtimes\Z^2$
 is $\{(x,0,0,x) \in G_A^{s,u}\rtimes\Z^2\mid x \in \bar{X}_A\},$
which is regarded as $\bar{X}_A$, 
the algebra $\WRA$ includes $C(\bar{X}_A)$ as a subalgebra in natural way.
There is a projection $E_A$ in the tensor product 
$\mathcal{O}_{A^t}\otimes\OA$ such that 
$\WRA$ is naturally isomorphic to $E_A(\mathcal{O}_{A^t}\otimes\OA)E_A$.
Hence  the algebra $\WRA$ might be regarded as a bilateral Cuntz--Krieger algebra.
The tensor product 
$\alpha^{A^t}_r\otimes\alpha^A_s$ of the
gauge actions
$\alpha^{A^t}_r$ on $\mathcal{O}_{A^t}$
and
$\alpha^A_s$ on $\OA$ 
gives rise to an action $\gamma^A_{(r,s)}$ of $(r,s) \in \T^2$
on $\WRA$.
It has been shown in \cite{MaPre2017} that the fixed point algebra 
$(\WRA)^{\delta^A}$ of $\WRA$ under
the diagonal action $\delta^A_t = \gamma^A_{(t,t)}, t \in \T$
is isomorphic to $\RA$.

In \cite{CK} (cf. \cite{Cu3}), Cuntz--Krieger proved that
the stabilized Cuntz--Krieger algebra $\SOA$ with
its diagonal $C^*$-subalgebra $\SDA$ of $\SOA$,
where $\C$ denotes the maximal abelian $C^*$-subalgebra of $\K$ 
consisting of diagonal operators on $\ell^2(\N)$,
 and the stabilized gauge action
$\alpha^A\otimes \id$ is invariant under topological conjugacy of 
the two-sided topological Markov shift $(\bar{X}_A,\bar{\sigma}_A)$
for irreducible non-permutation matrix $A$.
T. M. Carlsen and J. Rout  have recently proved in \cite{CR}
that the converse also holds
even for more general matrices without irreducibility and non-permutation.
As a consequence,
the stabilized Cuntz--Krieger algebra $\SOA$ with
its diagonal $C^*$-subalgebra $\SDA$ and the stabilized gauge action
$\alpha^A\otimes \id$ is a complete invariant 
of the topological conjugacy of the two-sided topological Markov shift.
Inspired by this fact, 
we will in this paper show that 
the Ruelle algebra $\RA$ with its subalgebra $C(\bar{X}_A)$ and the dual action
$\rho^A$ is a complete invariant of the two-sided topological Markov shift
$(\bar{X}_A,\bar{\sigma}_A)$.
We will also see that the $C^*$-algebra $\WRA$  
with its subalgebra $C(\bar{X}_A)$ 
and the action $\gamma^A$ of $\T^2$
is also a complete invariant of  topological conjugacy of 
$(\bar{X}_A,\bar{\sigma}_A)$.
We will show  the following theorem.
\begin{theorem}[{Theorem \ref{thm:main}}] \label{thm:1.1}
Let $A,B$ be irreducible, non-permutation matrices with entries in $\{0,1\}$.
The following six conditions are equivalent.
\begin{enumerate}
\renewcommand{\theenumi}{\roman{enumi}}
\renewcommand{\labelenumi}{\textup{(\theenumi)}}
\item
Topological Markov shifts 
$(\bar{X}_A,\bar{\sigma}_A)$ and
$(\bar{X}_B,\bar{\sigma}_B)$
 are topologically conjugate.
\item
Topological Markov shifts 
$(\bar{X}_A,\bar{\sigma}_A)$ and
$(\bar{X}_B,\bar{\sigma}_B)$
 are asymptotically  conjugate.
\item
There exists an isomorphism 
$\varphi: G_A^{a}\rtimes\Z \longrightarrow G_B^a\rtimes\Z$
of \'etale groupoids such that $d_B\circ\tilde{\varphi} = d_A$. 
\item
There exists an isomorphism 
$\tilde{\varphi}: G_A^{s,u}\rtimes\Z^2 \longrightarrow G_B^{s,u}\rtimes\Z^2$
of \'etale groupoids such that $c_B\circ\tilde{\varphi} = c_A$. 
\item
There exists an isomorphism 
$\Phi: \RA\longrightarrow\R_B$ of $C^*$-algebras
such that 
$\Phi(C(\bar{X}_A)) = C(\bar{X}_B)$
and
$\Phi\circ\rho^A_{t} =\rho^B_{t}\circ{\Phi}$
for
$t\in \T$.
\item
There exists an isomorphism 
$\tilde{\Phi}: \WRA\longrightarrow\widetilde{\R}_B$ of $C^*$-algebras
such that 
$\tilde{\Phi}(C(\bar{X}_A)) = C(\bar{X}_B)$
and
$\tilde{\Phi}\circ\gamma^A_{(r,s)} =\gamma^B_{(r,s)}\circ\tilde{\Phi}$
for
$(r,s)\in \T^2$.
\end{enumerate}
\end{theorem}
The equivalences among (ii), (iii) and (v) come from \cite{MaPre2017}.
The main assertion is the implication
(ii) $\Longrightarrow$ (i) which will be proved in Theorem \ref{thm:section2}.
Other implications will be seen in the proof of Theorem \ref{thm:main},
which are not tough tasks.

Since the algebra $\WRA$ is a unital, simple, purely infinite, nuclear $C^*$-algebra satisfying UCT,
its isomorphism class is completely determined by its K-theory date 
by a general classification theorem 
(\cite{Kirchberg}, \cite{Phillips}, \cite{Ro2}).
The K-groups $K_*(\WRA)$ are seen by the K\"{u}nneth formulas 
and the universal coefficient theorem
such that 
\begin{equation*}
K_0(\WRA) \cong KK^1(\OTA,\OA),\qquad K_1(\WRA) \cong KK(\OTA,\OA).
\end{equation*}
As a corollary of Theorem \ref{thm:1.1}, 
we know that the group
$K_0(\WRA)$ and the position of the class $[1_{\WRA}]$
of the unit $1_{\WRA}$ of $\WRA$ in $K_0(\WRA)$
is invariant under topological conjugacy of 
$(\bar{X}_A,\bar{\sigma}_A)$.
We see that
$(K_0(\WRA), [1_{\WRA}])$
is isomorphic to 
$(K_0(\OTA\otimes\OA), [E_A])$
and
the class $[E_A]$ of the projection $E_A$
actually lives  in the group
$K_0(\OTA)\otimes K_0(\OA)$.
We set
the vector
$e_i = [0,\dots,0,\overset{i}{1},0,\dots,0] \in \Z^N
$ for $i=1,\dots,N$.
We have the following theorem.
\begin{theorem}[{Theorem \ref{thm:Kinvariant}}]
Suppose that $A$ is an $N\times N$ irreducible, non-permutation matrix
with entries in $\{0,1\}$.
 The position $[E_A]$ of the projection $E_A$ in 
$K_0(\OTA)\otimes K_0(\OA)$ is 
$\sum_{i=1}^N[e_i]\otimes [e_i]$ in the group
$\Z^N/(\id-A)\Z^N \otimes
\Z^N/(\id-A^t)\Z^N.
$
Hence
it is invariant under topological conjugacy of the two-sided topological Markov shift
$(\bar{X}_A,\bar{\sigma}_A)$.
\end{theorem}
We put
$$
e_A=\sum_{i=1}^N[e_i]\otimes [e_i] \quad\text{ in }
\Z^N/(\id-A)\Z^N \otimes
\Z^N/(\id-A^t)\Z^N.
$$
We will actually see that the pair
$(\Z^N/(\id-A)\Z^N \otimes
\Z^N/(\id-A^t)\Z^N, e_A)
$
is a shift equivalence invariant
(Proposition \ref{prop:SEinvariant}, Proposition \ref{prop:NNSEinvariant}).
We will present an example of matrices 
$A=[A(i,j)]_{i,j=1}^N, B=[B(i,j)]_{i,j=1}^M$ such that 
$K_0(\OA) \cong K_0(\OB), \det(\id-A) = \det(\id - B)$,
but the invariants
$(\Z^N/(\id-A)\Z^N \otimes
\Z^N/(\id-A^t)\Z^N, e_A)
$
and
$
(\Z^M/(\id-B)\Z^M \otimes
\Z^M/(\id-B^t)\Z^M, e_B)
$
are different (Proposition \ref{prop:example}).
This shows that the invariant 
$(\Z^N/(\id-A)\Z^N \otimes\Z^N/(\id-A^t)\Z^N, e_A)$
is strictly stronger than the Bowen-Franks group 
$\Z^N/(\id-A)\Z^N$
and not invariant under flow equivalence.

J. Cuntz in \cite{Cu4} 
studied the homotopy groups 
$\pi_n(\End(\OA\otimes\K))$
of the space $\End(\OA\otimes\K)$ of endomorphisms
of the $C^*$-algebras $\OA\otimes\K.$
He proved that  natural maps 
$\epsilon_n : \pi_n(\End(\OA\otimes\K))\longrightarrow KK^n(\OA,\OA)$ 
yield  isomorphisms,
and
defined an element denoted by $\epsilon_1(\lambda^A)$
in $\Ext(\OA)\otimes K_0(\OA),$ where
$\lambda^A$ denotes the gauge action $\alpha^A$ on $\OA.$
His observation shows that the element $\epsilon_1(\lambda^A)$
is noting but the above element $e_A$ under the natural identification
between 
$\Ext(\OA)\otimes K_0(\OA)$
and
$\Z^N/(\id-A)\Z^N \otimes\Z^N/(\id-A^t)\Z^N.$
He already states in \cite{Cu4} that 
the position $\epsilon_1(\lambda^A)$ in 
$\Z^N/(\id-A)\Z^N \otimes\Z^N/(\id-A^t)\Z^N$
is invariant under topological conjugacy of the topological Markov shift
$(\bar{X}_A,\bar{\sigma}_A)$.

We will finally study  that 
KMS states for the diagonal action 
$\delta^A_t =\gamma^A_{(t,t)}$ on $\WRA$,
and prove the following theorem.
\begin{theorem}[{Theorem \ref{thm:KMS}}]
Assume that the matrix $A$ is aperiodic.
A KMS state on $\WRA$ for the action $\delta^A$ 
at the inverse temperature $\log \gamma$
exists if and only if $\gamma$ is the Perron--Frobenius eigenvalue $\beta$ of $A$.
The admitted KMS state is unique.
The restriction of the admitted KMS state to the subsalgebra $C(\bar{X}_A)$
is the state defined by the Parry measure on $\bar{X}_A$.
\end{theorem}
The Parry measure is the measure of maximal entropy
(cf. \cite{Walter}).
Since $\log \beta$ is the topological entropy of the Markov shift
$(\bar{X}_A,\bar{\sigma}_A)$, the inverse temperature expresses the entropy.
This exactly corresponds to the result obtained 
by Enomoto--Fujii--Watatani in \cite{EFW1984} 
on KMS states for the gauge action on the Cuntz--Krieger algebras 
$\OA$. 

Throughout the paper, we denote by $\Zp$ the set of nonnegative integers
and by $\N$ the set of positive integers.

This paper is a continuation of the paper \cite{MaPre2017}.
\section{Preliminaries}
We fix an irreducible, non-permutation matrix $A=[A(i,j)]_{i,j=1}^N$
with entries in $\{0,1\}$.
Let $\OA$ and $\OTA$
be the Cuntz--Krieger algebras 
for the matrices $A$ and its transpose $A^t$, respectively.
We may take generating partial isometries
$S_i, i=1,\dots,N$ of $\OA$
and
$T_i, i=1,\dots,M$ of $\OTA$
such that 
\begin{gather}
\sum_{i=1}^N S_i S_i^* =1, \qquad 
S_i^* S_i = \sum_{j=1}^N A(i,j)S_j S_j^*, \label{eq:CKA}\\
\sum_{i=1}^N T_i T_i^* =1, \qquad 
T_i^* T_i = \sum_{j=1}^N A^t(i,j)T_j T_j^*.\label{eq:CKTA}
\end{gather}
In the $C^*$-algebra 
$\OTA\otimes\OA$ of tensor product,
let us denote by 
$E_A$ the projection defined by
\begin{equation}
E_A = \sum_{i=1}^N T_i^* T_i\otimes S_i S_i^*.  \label{eq:EA}
\end{equation}
By using the relations \eqref{eq:CKA} and \eqref{eq:CKTA},
it is easy to see that 
$E_A = \sum_{i=1}^N T_i T_i^*\otimes S_i^* S_i.$
The $C^*$-algebra $\widetilde{R}_A$ is defined as the groupoid $C^*$-algebra
$C^*(G_A^{s,u}\rtimes\Z^2)$ which is realized as the $C^*$-algebra (\cite{MaPre2017}) 
\begin{equation*}
\WRA =E_A(\OTA\otimes\OA)E_A. 
\end{equation*}
The $C^*$-algebra $\WRA$ was denoted by $\R^{s,u}_A$ in \cite{MaPre2017}.
Since both the algebras
$\OTA, \OA$
are simple and purely infinite,
and
$\WRA\otimes\K$ is isomorphic to $\OTA\otimes\OA\otimes\K$,
the $C^*$-algebra $\WRA$ is simple and purely infinite
if $A$ is irreducible and non-permutation (cf. \cite[Proposition 5.5]{PutSp}).

Let $B_n(\bar{X}_A)$ be the set of admissible words in $\bar{X}_A$
of length $n$.
We set $B_*(\bar{X}_A) =\cup_{n=0}^\infty B_n(\bar{X}_A),$
where $B_0(\bar{X}_A)$ denotes the empty word.  
For a word $\xi=(\xi_1,\dots,\xi_k), \mu=(\mu_1,\dots,\mu_m) \in B_*(\bar{X}_A)$,
we denote by
$\bar{\xi} =(\xi_k,\dots,\xi_1)\in B_k(\bar{X}_{A^t})$ and set
$T_{\bar{\xi}} =T_{\xi_k}\cdots T_{\xi_1}$ and 
$S_\mu=S_{\mu_1}\cdots S_{\mu_m}.$
Let $\alpha^A, \alpha^{A^t}$ be the gauge actions of $\OA$ and $\OTA$, respectively,
which are defined by
\begin{equation*}
\alpha^A_t(S_i) = \exp(\sqrt{-1}t) S_i, \qquad
\alpha^{A^t}_t(T_i) = \exp(\sqrt{-1}t) T_i, \qquad i=1,\dots,N,\, 
t \in {\mathbb{R}}/2\pi\Z =\T.
\end{equation*}
The fixed point algebras
$(\OA)^{\alpha^A},(\OTA)^{\alpha^{A^t}}
$
of $\OA, \OTA$
under the gauge actions $\alpha^A, \alpha^{A^t}$
are known to be AF-algebras, which are denoted by
$\FA, \FTA$, respectively.

We first note the following facts which were seen in \cite{MaPre2017}.
\begin{proposition}\label{prop:2.1}\hspace{6cm}
\begin{enumerate}
\renewcommand{\theenumi}{\roman{enumi}}
\renewcommand{\labelenumi}{\textup{(\theenumi)}}
\item
The groupoid $C^*$-algebra $C^*(G_A^a)$ of the groupoid
$G_A^a$ is isomorphic to 
the $C^*$-subalgebra of $\FTA\otimes\FA$ 
defined by
\begin{align*}
C^*(\quad &T_{\bar{\xi}}T_{\bar{\eta}}^*\otimes S_\mu S_\nu^* \in \OTA\otimes\OA \mid \\
& \mu =(\mu_1,\dots,\mu_m), \nu =(\nu_1,\dots,\nu_n) \in B_*(\bar{X}_A),\\
& \bar{\xi} =(\xi_k,\dots,\xi_1), \bar{\eta} =(\eta_l,\dots,\eta_1) \in B_*(\bar{X}_{A^t}),\\ 
& A(\xi_k,\mu_1) = A(\eta_l,\nu_1) =1, k=l, m=n \quad ).
\end{align*}
\item
The $C^*$-algebra $\R_A$ is isomorphic to 
the $C^*$-subalgebra of $\OTA\otimes\OA$ 
defined by
\begin{align*}
C^*(\quad &T_{\bar{\xi}}T_{\bar{\eta}}^*\otimes S_\mu S_\nu^* \in \OTA\otimes\OA \mid \\
& \mu =(\mu_1,\dots,\mu_m), \nu =(\nu_1,\dots,\nu_n) \in B_*(\bar{X}_A),\\
& \bar{\xi} =(\xi_k,\dots,\xi_1), \bar{\eta} =(\eta_l,\dots,\eta_1) \in B_*(\bar{X}_{A^t}),\\ 
& A(\xi_k,\mu_1) = A(\eta_l,\nu_1) =1, \, k-l = n-m \quad).
\end{align*}
\item
The $C^*$-algebra $\widetilde{\R}_A$ is isomorphic to 
the $C^*$-subalgebra of $\OTA\otimes\OA$ 
defined by
\begin{align*}
C^*(\quad &T_{\bar{\xi}}T_{\bar{\eta}}^*\otimes S_\mu S_\nu^* \in \OTA\otimes\OA \mid \\
& \mu =(\mu_1,\dots,\mu_m), \nu =(\nu_1,\dots,\nu_n) \in B_*(\bar{X}_A),\\
& \bar{\xi} =(\xi_k,\dots,\xi_1), \bar{\eta} =(\eta_l,\dots,\eta_1) \in B_*(\bar{X}_{A^t}),\\ 
& A(\xi_k,\mu_1) = A(\eta_l,\nu_1) =1\quad ).
\end{align*}
\end{enumerate}
\end{proposition}
We note that for $i=1,\dots,N$ the identity
\begin{equation*}
T_i \otimes S_i^* 
=
\sum_{j,k=1}^N T_i T_j T_j^* \otimes S_k S_k^*S_i^*
=
\sum_{j,k=1}^N A(j,i)A(i,k) T_{ij} T_j^* \otimes S_k S_{ik}^*
\end{equation*}
holds.
Since $A(j,i)A(i,k) T_{ij} T_j^* \otimes S_k S_{ik}^*$ belongs to $\RA$,
we see that $T_i \otimes S_i^* $ and hence 
$T_i^* \otimes S_i $ belong to $\RA$.

Define the diagonal action $\delta^A$ on $\widetilde{R}_A$
by setting
\begin{equation*}
\delta^A_t = \alpha^{A^t}_t\otimes \alpha^A_t, \qquad t \in {\mathbb{R}}/2\pi\Z =\T.
\end{equation*}
Since $\delta^A_t(E_A) = E_A$,
the automorphisms
$\delta^A_t, t \in \T$ 
define an action of $\T$ on $\WRA$.
For 
\begin{gather*}
 \mu =(\mu_1,\dots,\mu_m), \quad \nu =(\nu_1,\dots,\nu_n) \in B_*(\bar{X}_A),\\
 \bar{\xi} =(\xi_k,\dots,\xi_1), \quad \bar{\eta} =(\eta_l,\dots,\eta_1) \in B_*(\bar{X}_{A^t})
\end{gather*}
satisfying
$ 
 A(\xi_k,\mu_1) = A(\eta_l,\nu_1) =1,
$
we see that 
\begin{equation*}
\delta^A_t(T_{\bar{\xi}}T_{\bar{\eta}}^*\otimes S_\mu S_\nu^*)
= \exp(\sqrt{-1}(k-l+m-n)t) T_{\bar{\xi}}T_{\bar{\eta}}^*\otimes S_\mu S_\nu^* 
\end{equation*}
so that the following lemma holds.
\begin{lemma}
Keep the above notation.
The element
$T_{\bar{\xi}}T_{\bar{\eta}}^*\otimes S_\mu S_\nu^*$ in $\WRA$ belongs to $\RA$
if and only if $k-l = n-m$.
\end{lemma}
Hence we have
\begin{proposition}[{\cite[Theorem 9.6]{MaPre2017}}]
The fixed point algebra 
$(\WRA)^{\delta^A}$
of $\WRA$ under $\delta^A$ is the asymptotic Ruelle algebra $\RA$.
\end{proposition}
As in \cite[Lemma 9.5]{MaPre2017}, the $C^*$-subalgebra of $C^*(G_A^a)$
generated by elements 
$T_{\bar{\xi}}T_{\bar{\xi}}^*\otimes S_\mu S_\mu^*,\,
\bar{\xi} =(\xi_k,\dots,\xi_1) \in B_*(\bar{X}_{A^t}),
\mu =(\mu_1,\dots,\mu_m) \in B_*(\bar{X}_A) 
$
with $A(\xi_k,\mu_1) =1$ is canonically isomorphic to the commutative $C^*$-algebra
$C(\bar{X}_A)$ of continuous functions on $\bar{X}_A$.
In what follows, we identify the subalgebra with the algebra $C(\bar{X}_A)$
so that $C(\bar{X}_A)$ is a $C^*$-subalgebra of $\RA$ and $\WRA$.
\section{Asymptotic conjugacy and topological conjugacy}

For $x=(x_n)_{n\in\Z} \in \bar{X}_A$,  we set
$x_+= (x_n)_{n=0}^\infty$ and
$x_-= (x_{-n})_{n=0}^\infty$. 
Let us denote by $X_A$ 
the compact Hausdorff space 
 of right infinite sequences 
$(x_i)_{i\in\Zp} \in \{1,\dots,N\}^{\Zp}$ 
satisfying $A(x_i,x_{i+1}) =1, i \in \Zp$.
The right one-sided topological Markov shift
$(X_A,\sigma_A)$
is defined by a topological dynamical system 
of shift transformation $\sigma_A((x_i)_{i\in\Zp}) = (x_{i+1})_{i\in\Zp}$   
on $X_A$.
For $x =(x_i)_{i\in\Zp} \in X_A$ and $k \in \Zp$,
we set
$x_{[k,\infty)} = \sigma_A^k(x) =(x_k, x_{k+1}, \dots ) \in X_A.$

In \cite{MaPre2017}, 
a notion of asymptotic  conjugacy in Smale spaces were introduced.
We apply the notion for topological Markov shifts and rephrase it in the following way.
\begin{definition}[{\cite{MaPre2017}}]
Two topological Markov shifts
$(\bar{X}_A,\bar{\sigma}_A)$
and
$(\bar{X}_B,\bar{\sigma}_B)$
are said to be {\it asymptotically conjugate}\/
if there exists a homeomorphism
$h:\bar{X}_A \longrightarrow \bar{X}_B$ satisfying the following three conditions
\begin{enumerate}
\renewcommand{\theenumi}{\roman{enumi}}
\renewcommand{\labelenumi}{\textup{(\theenumi)}}
\item
There exists a nonnegative integer $K \in \Zp$ 
such that
\begin{align}
\bar{\sigma}_B^{K+1}(h(x))_+ & = \bar{\sigma}_B^{K}(h(\bar{\sigma}_A(x)))_+ 
\quad \text{ for }  x \in \bar{X}_A,
\label{eq:asp1}\\
\bar{\sigma}_B^{-K+1}(h(x))_- & = \bar{\sigma}_B^{-K}(h(\bar{\sigma}_A(x)))_- 
\quad \text{ for } x \in \bar{X}_A,
\label{eq:asp2}\\
\bar{\sigma}_A^{K+1}(h^{-1}(y))_+ & = \bar{\sigma}_A^{K}(h^{-1}(\bar{\sigma}_B(y)))_+
\quad \text{ for }  y \in \bar{X}_B,
\label{eq:asp3}\\
\bar{\sigma}_A^{-K+1}(h^{-1}(y))_- & = \bar{\sigma}_A^{-K}(h^{-1}(\bar{\sigma}_B(y)))_-
\quad \text{ for } y \in \bar{X}_B.
\label{eq:asp4} 
\end{align}
\item
There exists a continuous function
$m_1: G_A^a\longrightarrow\Zp$ such that 
\begin{gather*}
\bar{\sigma}_B^{m_1(x,z)}(h(x))_+ 
=
\bar{\sigma}_B^{m_1(x,z)}(h(z))_+ 
\quad \text{ for } (x,z) \in G_A^a, \\
\bar{\sigma}_B^{-m_1(x,z)}(h(x))_-
=
\bar{\sigma}_B^{-m_1(x,z)}(h(z))_- 
\quad \text{ for } (x,z) \in G_A^a.
\end{gather*}
\item
There exists a continuous function
$m_2: G_B^a\longrightarrow\Zp$ such that 
\begin{gather*}
\bar{\sigma}_A^{m_2(y,w)}(h^{-1}(y))_+
=
\bar{\sigma}_A^{m_2(y,w)}(h^{-1}(w))_+ 
\quad \text{ for } (y,w) \in G_B^a, \\
\bar{\sigma}_A^{-m_2(y,w)}(h^{-1}(y))_-
=
\bar{\sigma}_A^{-m_2(y,w)}(h^{-1}(w))_- 
\quad \text{ for } (y,w) \in G_B^a.
\end{gather*}
\end{enumerate}
\end{definition}

Let $A=[A(i,j)]_{i,j=1}^N, B=[B(i,j)]_{i,j=1}^M$
be irreducible matrices with entries in $\{0,1\}$.
The following proposition is key in this section.
\begin{proposition}
If the topological Markov shifts 
$(\bar{X}_A,\bar{\sigma}_A)$ and
$(\bar{X}_B,\bar{\sigma}_B)$
are
asymptotically conjugate, 
then they are topologically conjugate.
\end{proposition}
\begin{proof}
Let $h:\bar{X}_A \longrightarrow \bar{X}_B$ be a homeomorphism
and $K \in \Zp$
a nonnegative integer  
 satisfying \eqref{eq:asp1}, \eqref{eq:asp2}, \eqref{eq:asp3}, \eqref{eq:asp4}.
We define two continuous maps
$h_+: \bar{X}_A\longrightarrow X_B$
and
$h^{-1}_+: \bar{X}_B\longrightarrow X_A$ by setting
\begin{align*}
h_+(x) & = \bar{\sigma}_B^{K}(h(x))_+, \qquad x \in \bar{X}_A, \\
h^{-1}_+(y) & = \bar{\sigma}_A^{K}(h^{-1}(y))_+,\qquad y \in \bar{X}_B. 
\end{align*}
It then follows that by \eqref{eq:asp1},
\begin{align*}
h_+(\bar{\sigma}_A(x))
& = \bar{\sigma}_B^{K}(h(\bar{\sigma}_A(x)))_+\\
& =\bar{\sigma}_B^{K+1}(h(x))_+ \\
& =[\bar{\sigma}_B^{K+1}(h(x))]_{[0,\infty)} \\
& =[\bar{\sigma}_B(h(x))]_{[K,\infty)}.
\end{align*}
On the other hand, 
\begin{align*}
\sigma_B(h_+(x))
& =\sigma_B( [\bar{\sigma}_B^{K}(h(x))]_{[0,\infty)})\\
& =[\bar{\sigma}_B^{K}(h(x))]_{[1,\infty)}\\
& =[h(x)]_{[K+1,\infty)}\\
& =[\bar{\sigma}_B(h(x))]_{[K,\infty)}.
\end{align*}
Therefore we have
\begin{equation*}
h_+(\bar{\sigma}_A(x)) = \sigma_B(h_+(x)) \qquad \text{ for } x \in \bar{X}_A.
\end{equation*} 
Hence the continuous map $h_+: \bar{X}_A\longrightarrow X_B$
is a sliding block code (cf. \cite{LM}) so that 
there exists a block map
$\Phi: B_{m+n+1}(\bar{X}_A) \longrightarrow \{1,2,\dots,M\}$ 
for some $m,n\in \Zp$ such that 
\begin{equation*}
h_+((x_i)_{i\in\Z}) = \Phi([x_{i-m},\dots,x_{i+n}])_{i\in \Zp} 
\qquad \text{ for } x=(x_i)_{i\in\Z} \in \bar{X}_A. 
\end{equation*}
Similarly we know that 
the continuous map $h^{-1}_+: \bar{X}_B\longrightarrow X_A$
satisfies
$h^{-1}_+(\bar{\sigma}_B(y)) = \sigma_A(h^{-1}_+(y))$
for
$y \in \bar{X}_B
$
 so that 
there exists a block map
$\Psi: B_{m'+n'+1}(\bar{X}_B) \longrightarrow \{1,2,\dots,N\}$ 
for some $m',n'\in \Zp$  such that 
\begin{equation*}
h^{-1}_+((y_i)_{i\in\Z}) = \Psi([y_{i-m'},\dots, y_{i+n'}])_{i\in \Zp} 
\qquad \text{ for } y=(y_i)_{i\in\Z} \in \bar{X}_B. 
\end{equation*}
By using these block maps
$\Phi: B_{m+n+1}(\bar{X}_A) \longrightarrow \{1,2,\dots,M\}$ 
and
$\Psi: B_{m'+n'+1}(\bar{X}_B) \longrightarrow \{1,2,\dots,N\},$  
we define two sliding block codes  
$\Phi_{\infty}:\bar{X}_A\longrightarrow \bar{X}_B$ and
$\Psi_{\infty}:\bar{X}_B\longrightarrow \bar{X}_A$
by setting
\begin{align*}
\Phi_{\infty}((x_i)_{i\in\Z}) & = \Phi([x_{i-m},\dots,x_{i+n}])_{i\in \Z} \in \bar{X}_B 
\qquad \text{ for } x=(x_i)_{i\in\Z} \in \bar{X}_A,\\
\Psi_\infty((y_i)_{i\in\Z}) & = \Psi([y_{i-m'},\dots, y_{i+n'}])_{i\in \Z} \in \bar{X}_A 
\qquad \text{ for } y=(y_i)_{i\in\Z} \in \bar{X}_B. 
\end{align*}
We note that 
\begin{align*}
\Phi_{\infty}((x_i)_{i\in\Z})_+ & = h_+((x_i)_{i\in\Z}) \in X_B 
\qquad \text{ for } x=(x_i)_{i\in\Z} \in \bar{X}_A,\\
\Psi_\infty((y_i)_{i\in\Z})_+ & = h^{-1}_+((y_i)_{i\in\Z}) \in {X}_A 
\qquad \text{ for } y=(y_i)_{i\in\Z} \in \bar{X}_B. 
\end{align*}
For $y = (y_i)_{i\in\Z} \in \bar{X}_B$, we have
\begin{align*}
[\Psi_\infty(y)]_{[K,\infty)}
= & [\bar{\sigma}_A^K(\Psi_\infty(y))]_{[0,\infty)} \\
= & [\Psi_\infty(\bar{\sigma}_B^K(y))]_{[0,\infty)} \\
= & h^{-1}_+(\bar{\sigma}_B^K(y)) \\
= & [ \bar{\sigma}_A^K(h^{-1}(\bar{\sigma}_B^K(y)))]_{[0,\infty)} \\
= & [h^{-1}(\bar{\sigma}_B^K(y))]_{[K,\infty)}. 
\end{align*}
As $\Phi_\infty$ is a sliding block code with memory $m$, 
the condition
$
[\Psi_\infty(y)]_{[K,\infty)} =[h^{-1}(\bar{\sigma}_B^K(y)))]_{[K,\infty)}$
implies
\begin{equation*}
[\Phi_\infty(\Psi_\infty(y))]_{[K+m,\infty)}
 =[\Phi_\infty(h^{-1}(\bar{\sigma}_B^K(y)))]_{[K+m,\infty)}.
\end{equation*}
It then follows that 
\begin{align*}
[\Phi_\infty(\Psi_\infty(y))]_{[K+m,\infty)}
= & [\Phi_\infty(h^{-1}(\bar{\sigma}_B^K(y)))]_{[K+m,\infty)} \\
= & [h_+(h^{-1}(\bar{\sigma}_B^K(y)))]_{[K+m,\infty)} \\
= & [(\bar{\sigma}_B^K\circ h) (h^{-1}(\bar{\sigma}_B^K(y)))]_{[K+m,\infty)} \\
= & [\bar{\sigma}_B^K(\bar{\sigma}_B^K(y))]_{[K+m,\infty)} \\
= & [\bar{\sigma}_B^{2K}(y)]_{[K+m,\infty)} 
\end{align*}
so that 
\begin{equation*}
[\Phi_\infty(\Psi_\infty(y))]_{[K+m,\infty)}
= [\bar{\sigma}_B^{2K}(y)]_{[K+m,\infty)} \quad \text{ for } y \in \bar{X}_B. 
\end{equation*}
Since $\Phi_\infty\circ\Psi_\infty$ is a sliding block code, we obtain that
\begin{equation*}
\Phi_\infty \circ\Psi_\infty = \bar{\sigma}_B^{2K}.
\end{equation*}
Hence 
$\Phi_\infty$ is surjective.
Similarly we know that 
$\Psi_\infty \circ\Phi_\infty = \bar{\sigma}_A^{2K}
$ so that 
$\Phi_\infty$ is injective.
Therefore we have a topological conjugacy
$\Phi_{\infty}:\bar{X}_A\longrightarrow \bar{X}_B$.
\end{proof}
We remark that the above proof needs only the equalities
\eqref{eq:asp1} and \eqref{eq:asp3}.

We thus conclude the following.
\begin{theorem}\label{thm:section2}
Two topological Markov shifts 
$(\bar{X}_A,\bar{\sigma}_A)$ and
$(\bar{X}_B,\bar{\sigma}_B)$
are
asymptotically conjugate if and only if they are topologically conjugate.
\end{theorem}
\begin{proof}
It is direct to see that topological conjugacy implies asymptotic conjugacy.
Hence the assertion follows from the preceding proposition.
\end{proof}

\section{Conjugacy, groupoid isomorphism and $C^*$-algebras}
We consider the groupoid 
$G_A^{s,u}\rtimes\Z^2$ and its $C^*$-algebra written  $\WRA$.
Recall that an action $\gamma^A$ of $\T^2$
 on $\WRA =E_A(\OTA\otimes\OA)E_A$ is defined  by setting
\begin{equation*}
\gamma^A_{(r,s)} = \alpha^{A^t}_{r}\otimes\alpha^A_s \quad \text{ on }
\OTA\otimes\OA \,\text{ for } (r,s) \in \T^2.
\end{equation*}
Since $\gamma^A_{(r,s)}(E_A) = E_A$,
we have an action $\gamma^A$ of $\T^2$ on $\WRA$,
which defines two kinds of actions of $\T$ on $\WRA$
such that 
$$
\delta^A_t = \gamma^A_{(t,t)} \quad \text{ and }
\quad
\rho^A_t = \gamma^A_{(-\frac{t}{2},\frac{t}{2})}
\quad\text{ for }
t \in\T.
$$ 
We regard the groupoid $C^*$-algebra $C^*(G_A^a\rtimes\Z)$
as the $C^*$-crossed product $C^*(G_A^a)\rtimes\Z$
in a natural way.
Let us denote by $\hat{\sigma}^A$ 
the dual action on $C^*(G_A^a)\rtimes\Z$.
In the following lemma, the $C^*$-algebra $\WRA$
is regarded as a $C^*$-subalgebra of $\OTA\otimes\OA$ as in 
Proposition \ref{prop:2.1} (ii).
\begin{lemma}
There exists an isomorphism 
$\Psi:C^*(G_A^a)\rtimes\Z\longrightarrow \RA$
such that 
\begin{equation*}
\Psi(C(\bar{X}_A))= C(\bar{X}_A) \quad
\text{ and }
\quad
\Psi \circ \hat{\sigma}^A_t = \rho^A_t \circ\Psi, \quad t \in \T. 
\end{equation*} 
\end{lemma} 
\begin{proof}
Let $U_A$ be the unitary in $\RA$
defined by 
$U_A = \sum_{i=1}^N T_i^* \otimes S_i$.
As in \cite[Proposition 9.9]{MaPre2017}, $\Ad(U_A)$ corresponds to the shift 
operation on $C(\bar{X}_A)$.
Since
\begin{equation*}
\rho^A_t(U_A) 
= \sum_{i=1}^N \alpha_{-\frac{t}{2}}(T_i^*)
  \otimes \alpha_{\frac{t}{2}}(S_i)
= \exp(\sqrt{-1}t)\sum_{i=1}^N T_i^* \otimes S_i
= \exp(\sqrt{-1}t)U_A,
 \end{equation*}
we have the assertion.
\end{proof}
We have the following main result of the paper.
\begin{theorem}\label{thm:main}
The following six conditions are equivalent.
\begin{enumerate}
\renewcommand{\theenumi}{\roman{enumi}}
\renewcommand{\labelenumi}{\textup{(\theenumi)}}
\item
Topological Markov shifts 
$(\bar{X}_A,\bar{\sigma}_A)$ and
$(\bar{X}_B,\bar{\sigma}_B)$
 are topologically conjugate.
\item
Topological Markov shifts 
$(\bar{X}_A,\bar{\sigma}_A)$ and
$(\bar{X}_B,\bar{\sigma}_B)$
 are asymptotically conjugate.
\item
There exists an isomorphism 
$\varphi: G_A^{a}\rtimes\Z \longrightarrow G_B^a\rtimes\Z$
of \'etale groupoids such that $d_B\circ\tilde{\varphi} = d_A$. 
\item
There exists an isomorphism 
$\tilde{\varphi}: G_A^{s,u}\rtimes\Z^2 \longrightarrow G_B^{s,u}\rtimes\Z^2$
of \'etale groupoids such that $c_B\circ\tilde{\varphi} = c_A$. 
\item
There exists an isomorphism 
$\Phi: \RA\longrightarrow\R_B$ of $C^*$-algebras
such that 
$\Phi(C(\bar{X}_A)) = C(\bar{X}_B)$
and
$\Phi\circ\rho^A_{t} =\rho^B_{t}\circ{\Phi}$
for
$t\in \T$.
\item
There exists an isomorphism 
$\tilde{\Phi}: \WRA\longrightarrow\widetilde{\R}_B$ of $C^*$-algebras
such that 
$\tilde{\Phi}(C(\bar{X}_A)) = C(\bar{X}_B)$
and
$\tilde{\Phi}\circ\gamma^A_{(r,s)} =\gamma^B_{(r,s)}\circ\tilde{\Phi}$
for
$(r,s)\in \T^2$.
\end{enumerate}
\end{theorem}
\begin{proof}
The equivalence between (i) and (ii) is proved in Theorem \ref{thm:section2}.

The equivalences among (ii), (iii) and (v) are shown in \cite{MaPre2017}.


We will prove the three implications 
(i) $\Longrightarrow$ (iv), 
(iv) $\Longrightarrow$ (vi),
(vi) $\Longrightarrow$ (v).

(i) $\Longrightarrow$ (iv): 
Suppose that there exists a topological conjugacy
$h:\bar{X}_A\longrightarrow \bar{X}_B$ so that
$h\circ\bar{\sigma}_A = \bar{\sigma}_B\circ h$.  
For $(x,p,q,z) \in G_A^{s,u}\rtimes\Z^2$,
the conditions 
$(\bar{\sigma}^p_A(x),z)\in G_A^s$
and
$(\bar{\sigma}^q_A(x),z)\in G_A^u$
 imply
$(\bar{\sigma}^p_B(h(x)),h(z))\in G_B^s$
and
$(\bar{\sigma}^q_A(h(x)),h(z))\in G_B^u$,
so that 
we have
$(h(x),p,q,h(z))\in G_B^{s,u}\rtimes\Z^2$.
It is routine to show that the correspondence
$$
\tilde{\varphi}:(x,p,q,z) \in G_A^{s,u}\rtimes\Z^2
\longrightarrow
(h(x),p,q,h(z))\in G_B^{s,u}\rtimes\Z^2
$$
yields an isomorphism of \'etale groupoids.
It is then clear that 
$c_B\circ\tilde{\varphi}=c_A$. 
This shows the condition (iv).

(iv) $\Longrightarrow$ (vi):
Suppose that 
there exists an isomorphism 
$\tilde{\varphi}: G_A^{s,u}\rtimes\Z^2 \longrightarrow G_B^{s,u}\rtimes\Z^2$
of \'etale groupoids such that $c_B\circ\tilde{\varphi} = c_A$.
Since the both groupoids 
$G_A^{s,u}\rtimes\Z^2$ and $G_B^{s,u}\rtimes\Z^2$
are amenable and \'etale
by \cite[Proposition 7.2 and Lemma 7.3]{MaPre2017}, 
the $C^*$-algebras
$\WRA$ and $\widetilde{\R}_B$
are represented on the Hilbert $C^*$-modules 
$\ell^2(G_A^{s,u}\rtimes\Z^2)$
and
$\ell^2(G_B^{s,u}\rtimes\Z^2)$,
respectively as in \cite{MaPre2017}.
As $\tilde{\varphi}: G_A^{s,u}\rtimes\Z^2 \longrightarrow G_B^{s,u}\rtimes\Z^2$
is an isomorphism of \'etale groupoids,
there exist a homeomorphism
$h:\bar{X}_A\longrightarrow \bar{X}_B$
and
a continuous groupoid homomorphism 
$c: G_A^{s,u}\rtimes\Z^2 \longrightarrow \Z^2$
such that 
$$
\varphi(x,p,q,z) = (h(x), c(x,p,q,z),h(z)),\qquad (x,p,q,z) \in G_A^{s,u}\rtimes\Z^2.
$$
The condition
$c_B\circ\tilde{\varphi} = c_A$
forces us to hold the equality
$c(x,p,q,z) =(p,q)$
so that we have
$$
\varphi(x,p,q,z) = (h(x), p,q,h(z)),\qquad (x,p,q,z) \in G_A^{s,u}\rtimes\Z^2.
$$
Let us consider the unitaries 
$V_h: \ell^2(G_B^{s,u}\rtimes\Z^2) \longrightarrow \ell^2(G_A^{s,u}\rtimes\Z^2)
$
and
$V_{h^{-1}}: \ell^2(G_A^{s,u}\rtimes\Z^2) \longrightarrow \ell^2(G_B^{s,u}\rtimes\Z^2)
$
by setting
\begin{align*}
[V_h\zeta](x,p,q,z) 
=& \zeta(h(x),p,q, h(z)) \text{ for }\zeta \in  \ell^2(G_B^{s,u}\rtimes\Z^2),
(x,p,q,z) \in G_A^{s,u}\rtimes\Z^2, \\
[V_{h^{-1}}\xi](y,m,n,w) 
=& \xi(h^{-1}(y),m,n, h^{-1}(w)) \text{ for }\xi \in  \ell^2(G_A^{s,u}\rtimes\Z^2),
(y,m,n,w) \in G_B^{s,u}\rtimes\Z^2.
\end{align*}
Put
$ \tilde{\Phi}= \Ad(V_h)$
which satisfies
$\tilde{\Phi}(C_c(G_A^{s,u}\rtimes\Z^2)) = C_c(G_B^{s,u}\rtimes\Z^2)
$
so that 
$ \tilde{\Phi}(\WRA) =\widetilde{\R}_B.$
Since
$\bar{X}_A, \bar{X}_B$ are identified
with the unit spaces
\begin{align*}
(G_A^{s,u}\rtimes\Z^2)^\circ 
=&  \{(x,0,0,x) \in G_A^{s,u}\rtimes\Z^2\mid x \in \bar{X}_A \}, \\
(G_B^{s,u}\rtimes\Z^2)^\circ 
=&  \{(y,0,0,y) \in G_B^{s,u}\rtimes\Z^2\mid y \in \bar{X}_B \}, 
\end{align*}
respectively,
we easily knows that 
$\tilde{\Phi}(C(\bar{X}_A)) = C(\bar{X}_B).$
It is also direct to see that the identity 
$\tilde{\Phi}\circ\gamma^A_{(r,s)} =\gamma^B_{(r,s)}\circ\tilde{\Phi}$
for
$(r,s)\in \T^2$
holds,
because of the equality
$c_B\circ\tilde{\varphi} = c_A.$
 

(vi) $\Longrightarrow$ (v):
Suppose that 
there exists an isomorphism 
$\tilde{\Phi}: \WRA\longrightarrow\widetilde{\R}_B$ of $C^*$-algebras
such that 
$\tilde{\Phi}(C(\bar{X}_A)) = C(\bar{X}_B)$
and
$\tilde{\Phi}\circ\gamma^A_{(r,s)} =\gamma^B_{(r,s)}\circ\tilde{\Phi}$
for
$(r,s)\in \T^2$.
As the action $\delta^A_t = \gamma_{(t,t)}^A$ of $t \in \T$ act on $\WRA$
and its fixed point algebra
$(\WRA)^{\delta^A}$ is $\RA$.
Let us denote by
$\Phi$ the restriction of $\tilde{\Phi}$ 
to the fixed point algebra $\RA$.
It induces an isomorphism
$\Phi:\RA\longrightarrow \R_B$. 
Then it is clear that the action 
$\rho^A_t = \gamma_{(-\frac{t}{2},\frac{t}{2})}^A$ on $\RA$ satisfies
 $\Phi\circ\rho^A_t = \rho^B_t\circ\Phi$.
 This shows the condition (v). 
\end{proof}

\section{K-theoretic invariants}
By using Theorem \ref{thm:main}, 
 the isomorphism classes of the $C^*$-algebras 
$\RA$ and $\WRA$ are invariant under topological conjugacy of two-sided topological Markov shifts.
Concerning the asymptotic Ruelle algebra $\RA$,
its K-group formula has been obtained by Putnam 
\cite[p.129]{Putnam1} (cf. \cite{Holton}, \cite{KilPut}).
We focus on studying the K-group $K_0(\WRA)$
of the latter algebra $\WRA$.
Under the assumption that the matrix $A$ is irreducible and non-permutation,
the algebra $\WRA$ is a unital, simple, purely infinite, nuclear $C^*$-algebra satisfying UCT,
so that its isomorphism class is completely determined by its K-theory date by
a general classification theory of Kirchberg (\cite{Kirchberg}) and Phillips (\cite{Phillips}).
Hence the following is a corollary of Theorem \ref{thm:main}.
\begin{proposition}
The pair $(K_0(\WRA), [1_{\WRA}])$ of the $K_0$-group of $\WRA$ and
the position of the unit $1_{\WRA}$ of $\WRA$ in $K_0(\WRA)$
is invariant under topological conjugacy of two-sided topological Markov shift
$(\bar{X}_A,\bar{\sigma}_A)$.  
\end{proposition}
Recall that the projection $E_A$ is defined in \eqref{eq:EA}. 
We have
\begin{proposition}\label{prop:KunitEA}
There exists an isomorphism
$\Phi: \WRA\otimes\K \longrightarrow\OTA\otimes\OA\otimes\K$
such that the induced isomorphism
$\Phi_*: K_0(\WRA)\longrightarrow K_0(\OTA\otimes\OA)$
satisfies
$\Phi_*([1_{\WRA}]) = [E_A].$
\end{proposition}
\begin{proof}
Since the $C^*$-algebra $\OTA\otimes\OA$ is unital and simple,
the projection $E_A$ in \eqref{eq:EA}
is a full projection in $\OTA\otimes\OA$,
Brown's theorem \cite{Brown} tells us that 
there exists an isometry
$v_A$ in the multiplier algebra $M(\OTA\otimes\OA\otimes\K)$
of $\OTA\otimes\OA\otimes\K$
such that 
$v_A^* v_A = 1$ and 
$v_A v_A^* = E_A\otimes 1.$
Define an isomorphism
$\Phi:\WRA\otimes\K \longrightarrow\OTA\otimes\OA\otimes\K$
by $\Phi =\Ad(v_A^*)$.
Let $p_0$ be a rank one projection in $\K$.
We then have
\begin{equation*}
\Phi_*([1_{\WRA}]) 
=\Phi_*([E_A\otimes p_0]) 
=[v_A^*(E_A\otimes p_0)v_A]
=[E_A\otimes p_0]  
=[E_A]  
\end{equation*}
in $K_0(\OTA\otimes\OA)$.
\end{proof}
Hence 
the position
$[E_A]$ in $K_0(\OTA\otimes\OA)$
as well as the group $K_0(\OTA\otimes\OA)$
is invariant under topological conjugacy of topological Markov shift
$(\bar{X}_A,\bar{\sigma}_A)$.
By the K\"{u}nneth formulas \cite{RS} of the K-groups of the tensor product $C^*$-algebras,
we know that 
\begin{gather*}
K_0(\OTA\otimes\OA) \cong
(K_0(\OTA)\otimes K_0(\OA)) \oplus (K_1(\OTA)\otimes K_1(\OA)), \\
K_1(\OTA\otimes\OA) \cong
(K_0(\OTA)\otimes K_1(\OA)) \oplus (K_1(\OTA)\otimes K_0(\OA))
\oplus 
\Tor^{\Z}_1(K_0(\OTA), K_0(\OA)).
\end{gather*}
By the universal coefficient theorem for KK-groups, 
the K-group $K_i(\OTA\otimes\OA)$ is isomorphic to the KK-group
$KK^{i+1}(\OTA\otimes\OA)$ for $i=0,1,$
we see that  
\begin{equation*}
K_0(\WRA) \cong KK^1(\OTA,\OA),\qquad K_1(\WRA) \cong KK(\OTA,\OA).
\end{equation*}
Since $K_0(\OTA)$ is isomorphic to $K_0(\OA)$ and
$K_1(\OA)$ is the torsion free part of $K_0(\OA),$
the groups 
$K_i(\OTA\otimes\OA), i=0,1$
do not have any further information  than the group $K_0(\OA)$
by the above K\"{u}nneth formulas.
As
$K_0(\OA)=\Z^N/(\id-A^t)\Z^N,$
it is a direct sum
$\Z^n \oplus T_A$
of its torsion free part $\Z^n$
and its torsion part
$T_A=\Z/m_1\Z\oplus\cdots\oplus\Z/m_k\Z$, where 
$m_i|m_{i+1}$ with $m_i \ge 2, \,i=1,\dots,k-1.$
It is easy to see that
\begin{align*}
 &\Z^N/(\id-A)\Z^N \otimes\Z^N/(\id-A^t)\Z^N \\
\cong & \Z^{n^2} \oplus (T_A)^{n} \oplus (T_A)^{n} \oplus (T_A\otimes T_A) \\
\cong & \Z^{n^2} \oplus (\Z/m_1\Z)^{2n+2k-1} \oplus (\Z/m_2\Z)^{2n+2k-3} \oplus
\cdots\oplus(\Z/m_k\Z)^{2n+1}.
\end{align*}
Hence the groups
$K_i(\OTA\otimes\OA), i=0,1$
give us the same information as the group $K_0(\OA)$.


The position
$[E_A]$ in $K_0(\OTA\otimes\OA)$
 however gives us more information than the group
$K_0(\OA)$.
In the above  K\"{u}nneth formula for 
$K_0(\OTA\otimes\OA)$, the element
$[E_A]$ lives in $K_0(\OTA)\otimes K_0(\OA)$
as the element
$\sum_{i=1}^N[T_i^*T_i]\otimes [S_iS_i^*]
$
by definition of $E_A$.
Therefore  
the position $[E_A]$ of the projection $E_A$ in 
$K_0(\OTA)\otimes K_0(\OA)$ is invariant under topological conjugacy of
$(\bar{X}_A,\bar{\sigma}_A)$.
We set
the vector
$e_i = [0,\dots,0,\overset{i}{1},0,\dots,0]
$ for $i=1,\dots,N$.
We rephrase the above fact with the following theorem.
\begin{theorem}\label{thm:Kinvariant}
The position $[E_A]$ of the projection $E_A$ in 
$K_0(\OTA)\otimes K_0(\OA)$ is invariant under topological conjugacy of
$(\bar{X}_A,\bar{\sigma}_A)$.
Hence the position
$\sum_{i=1}^N[e_i]\otimes [e_i]$ in the group
$\Z^N/(\id-A)\Z^N \otimes
\Z^N/(\id-A^t)\Z^N
$
is invariant under topological conjugacy of
$(\bar{X}_A,\bar{\sigma}_A)$.
\end{theorem}
\begin{proof}
We give its precise proof by  matrix method.
Let $A=[A(i,j)]_{i,j=1}^N, B=[B(i,j)]_{i,j=1}^M$
be irreducible non-permutation matrices such that the two-sided topological Markov shifts
$(\bar{X}_A,\bar{\sigma}_A)$ and
$(\bar{X}_A,\bar{\sigma}_A)$ are topological conjugate.
By William's theorem \cite{Williams},
the matrices $A,B$ are strong shift equivalent,
and hence we may assume that 
there exists two rectangular nonnegative integer matrices $C, D$ such that 
$A = CD, B = DC$.
By \cite[Theorem 4.6]{MaDocMath2017},
there exists an isomorphism
$\Phi:\SOA\longrightarrow \SOB$  
of $C^*$-algebras
such that the diagram
\begin{equation*}
\begin{CD}
K_0(\OA) @>\Phi_* >> K_0(\OB) \\
@V{\epsilon_A }VV  @VV{\epsilon_B}V \\
\Z^N/{(\id - A^{t})\Z^N} @> m_{C^t} >> \Z^M/{(\id - B^{t})\Z^M} 
\end{CD} 
\end{equation*}
commutes,
where $m_{C^t}$ is the isomorphism induced by multiplying the matrix
$C^t$ from the left
and $\epsilon_A: K_0(\OA) \rightarrow \Z^N/{(\id - A^{t})\Z^N}$
is an isomorphism defined by
$\epsilon_A([S_iS_i^*] ) = [e_i]$ the class of the vector $e_i$ in $\Z^N$.
Since the identities 
$A^t = D^t C^t, B^t = C^tD^t$
also hold, we similarly have  
an isomorphism
$\Phi^t:\SOTA\longrightarrow \SOTB$  
of $C^*$-algebras
such that the diagram 
\begin{equation*}
\begin{CD}
K_0(\OTA) @>\Phi^t_* >> K_0(\OTB) \\
@V{\epsilon_{A^t} }VV  @VV{\epsilon_{B^t}}V \\
\Z^N/{(\id - A)\Z^N} @> m_{D} >> \Z^M/{(\id - B)\Z^M} 
\end{CD} 
\end{equation*}
commutes.
We then have a commutative diagram:
\begin{equation*}
\begin{CD}
K_0(\OTA) \otimes K_0(\OA) @>\Phi^t_*\otimes \Phi_* >> K_0(\OTB)\otimes K_0(\OB) \\
@V{\epsilon_{A^t}\otimes\epsilon_A }VV  @VV{\epsilon_{B^t}\otimes\epsilon_B}V \\
\Z^N/{(\id - A)\Z^N}\otimes\Z^N/{(\id - A^{t})\Z^N} @> m_{D}\otimes m_{C^t} >> 
\Z^M/{(\id - B)\Z^M} \otimes\Z^M/{(\id - B^{t})\Z^M}.
\end{CD} 
\end{equation*}
We note that 
\begin{align*}
 \sum_{i=1}^N \epsilon_{A^t}([T_i^*T_i]) \otimes \epsilon_A([S_i S_i^*]) 
=&
\sum_{i=1}^N \epsilon_{A^t}([T_iT_i^*]) \otimes \epsilon_A([S_i S_i^*])\\
=&
\sum_{i=1}^N [e_i] \otimes [e_i] \quad 
\text{ in } \Z^N/{(\id - A)\Z^N}\otimes\Z^N/{(\id - A^{t})\Z^N},
\end{align*}
and set the specific element as 
\begin{equation}
e_A =\sum_{i=1}^N [e_i] \otimes [e_i] \quad 
\text{ in } \Z^N/{(\id - A)\Z^N}\otimes\Z^N/{(\id - A^{t})\Z^N}. \label{eq:eA}
\end{equation}
We will show that 
$(m_{D}\otimes m_{C^t})(e_A) = e_B.$
In the computation below, 
the vectors $e_i$, and $f_j$ 
denote the $N\times 1$ matrix
in $\Z^N$
whose $i$th component is one and zero elsewhere,
and
the $M\times 1$ matrix
in $\Z^M$
whose $j$th component is one and zero elsewhere,
respectively. 
We have
\begin{align*}
& \sum_{i=1}^N De_i \otimes C^t e_i \\
=& \sum_{i=1}^N 
{
\begin{bmatrix}
D(1,i)\\
D(2,i)\\
\vdots\\
D(M,i)
\end{bmatrix}
\otimes
\begin{bmatrix}
C(i,1)\\
C(i,2)\\
\vdots\\
C(i,M)
\end{bmatrix}}
 = \sum_{i=1}^N 
{\begin{bmatrix}
D(1,i)\\
D(2,i)\\
\vdots\\
D(M,i)
\end{bmatrix}
\otimes
\sum_{j=1}^M
C(i,j)f_j
} \\
=& 
\sum_{i=1}^N 
\sum_{j=1}^M
{\begin{bmatrix}
D(1,i)C(i,j)\\
D(2,i)C(i,j)\\
\vdots\\
D(M,i)C(i,j)
\end{bmatrix}
\otimes
f_j
} 
= 
\sum_{j=1}^M
{\begin{bmatrix}
\sum_{i=1}^N D(1,i)C(i,j)\\
\sum_{i=1}^N D(2,i)C(i,j)\\
\vdots\\
\sum_{i=1}^N D(N,i)C(i,j)
\end{bmatrix}
\otimes
f_j
} \\
=& 
\sum_{j=1}^M
{
\begin{bmatrix}
B(1,j)\\
B(2,j)\\
\vdots\\
B(N,j)
\end{bmatrix}
\otimes
f_j
= \sum_{j=1}^M Bf_j \otimes f_j.
} 
\end{align*}
Hence we have
\begin{equation}
 \sum_{i=1}^N De_i \otimes C^t e_i 
  -
  \sum_{j=1}^M f_j\otimes f_j 
=
\sum_{j=1}^M
(B -\id)f_j
\otimes f_j  \label{eq:DCB}
\end{equation}
so that 
\begin{equation*}
 (m_{D}\otimes m_{C^t})(e_A)  
= \sum_{i=1}^N [De_i] \otimes [C^te_i]
= \sum_{j=1}^M [f_j] \otimes [f_j] = e_B
\end{equation*}
thus proving the theorem.
\end{proof}
\begin{remark} \hspace{10cm}
\begin{enumerate}
\renewcommand{\theenumi}{\roman{enumi}}
\renewcommand{\labelenumi}{\textup{(\theenumi)}}
\item
The pair 
$
(\Z^N/(\id-A)\Z^N \otimes
\Z^N/(\id-A^t)\Z^N, e_A)
$
is a complete invariant for the isomorphism class of the 
$C^*$-algebra $\WRA$,
because the group structure of
$
\Z^N/(\id-A)\Z^N \otimes
\Z^N/(\id-A^t)\Z^N$
determines the groups $K_i(\OA), K_i(\OTA), i=0,1$
and also the pair determines the position
$[E_A]$ in $K_0(\OTA\otimes\OA)$.
Hence by Proposition \ref{prop:KunitEA},
the pair
$(K_0(\WRA), [E_A])$ and the group
$K_1(\WRA)$ are determined by
the pair 
$
(\Z^N/(\id-A)\Z^N \otimes
\Z^N/(\id-A^t)\Z^N, e_A).
$
\item
Since the projection $E_A$ is regarded as an element of the $C^*$-algebra
$\FTA\otimes\FA$ such that 
$C^*(G_A^a) = E_A(\FTA\otimes\FA)E_A$,
we have another topological conjugacy invariant
$(K_0(\FTA)\otimes K_0(\FA), [E_A])$
the position $[E_A]$ in the group $K_0(\FTA)\otimes K_0(\FA)$.
We will discuss this kind of invariants in 
\cite{MaPre2018a}.
\item
J. Cuntz in \cite{Cu4} 
studied the homotopy groups 
$\pi_n(\End(\OA\otimes\K))$
of the space of endomorphisms
$\End(\OA\otimes\K)$
of the $C^*$-algebras $\OA\otimes\K.$
He proved that  natural maps 
$\epsilon_n : \pi_n(\End(\OA\otimes\K))\longrightarrow KK^n(\OA,\OA)$ 
yield  isomorphisms,
and
defined an element denoted by $\epsilon_1(\lambda^A)$
in $\Ext(\OA)\otimes K_0(\OA),$ where
$\lambda^A$ denotes the gauge action $\alpha^A$ on $\OA.$
By the Kaminker--Putnam's K-theoretic duality between
$\Ext(\OA)$ and $K_0(\OTA)$ (\cite{KamPut}),
the element $\epsilon_1(\lambda^A)$ is regarded as an element in 
$K_0(\OTA)\otimes K_0(\OA)$.
Cuntz's observation in \cite{Cu4} shows that the element $\epsilon_1(\lambda^A)$
is noting but the above element $e_A$ under the identification
between 
$\Ext(\OA)\otimes K_0(\OA)$
and
$\Z^N/(\id-A)\Z^N \otimes\Z^N/(\id-A^t)\Z^N.$
He already states in \cite{Cu4} that 
the position $\epsilon_1(\lambda^A)$ in 
$\Z^N/(\id-A)\Z^N \otimes\Z^N/(\id-A^t)\Z^N$
is invariant under topological conjugacy of the topological Markov shift
$(\bar{X}_A,\bar{\sigma}_A)$. 
\end{enumerate}
\end{remark}

In \cite{Williams}, 
Williams introduced an equivalence relation in matrices called shift equivalence.
It is weaker than strong shift equivalence.
The shift equivalence relation  has been playing crucial r\^ole in the classification theory 
of symbolic dynamical systems (cf. \cite{LM}).   
Two matrices $A, B$ are said to be shift equivalent  
if there exist a positive integer  $\ell$ and rectangular nonnegative integer matrices $R,S$ such that 
\begin{equation}
AR = RB,\qquad SA = BS,\qquad A^{\ell} = RS,\qquad B^{\ell} = SR. \label{eq:SE}
\end{equation}
In the proof of the above theorem, we notice that 
 the following proposition holds.
\begin{proposition}\label{prop:SEinvariant}
The pair
$
(\Z^N/(\id-A)\Z^N \otimes
\Z^N/(\id-A^t)\Z^N, e_A)
$
is invariant under shift equivalence.
\end{proposition}
\begin{proof}
Suppose that matrices 
$A=[A(i,j)]_{i,j=1}^N, B=[B(i,j)]_{i,j=1}^M$ are shift equivalent.
Let $\ell$ be a positive integer and 
$R,S$ rectangular nonnegative integer matrices satisfying \eqref{eq:SE}.
Then the map
$m_S:\Z^N/(\id-A)\Z^N\longrightarrow \Z^M/(\id-B)\Z^M$
defined by the left multiplication of the matrix $S$ 
yields an isomorphism of the abelian groups.
We similarly see that 
$m_{R^t}:\Z^N/(\id-A^t)\Z^N\longrightarrow \Z^M/(\id-B^t)\Z^B$
defined by the left multiplication of the matrix $R^t$ 
yields an isomorphism of the abelian groups.
A similar computation proving the equality
\eqref{eq:DCB} in the proof of the preceding theorem
shows that the equality
\begin{equation*}
 \sum_{i=1}^N Se_i \otimes R^t e_i 
  -
  \sum_{j=1}^M f_j\otimes f_j 
=
\sum_{j=1}^M
(SR -\id)f_j
\otimes
f_j
=
\sum_{j=1}^M
(B^{\ell} -\id)f_j
\otimes
f_j
\end{equation*}
holds.
As $B^{\ell} -\id = (B-\id)(B^{\ell-1} +\cdots+ B + \id)$, 
we know that
$(m_S\otimes m_{R^t})(e_A) = e_B$
so that the map
$$
m_S\otimes m_{R^t}:\Z^N/(\id-A)\Z^N \otimes\Z^N/(\id-A^t)\Z^N
\longrightarrow
\Z^M/(\id-B)\Z^M \otimes\Z^M/(\id-B^t)\Z^M
$$
gives rise to an isomorphism between
$(\Z^N/(\id-A)\Z^N \otimes\Z^N/(\id-A^t)\Z^N, e_A)
$
and
$(\Z^M/(\id-B)\Z^M \otimes\Z^M/(\id-B^t)\Z^M, e_B).
$
\end{proof}

We will present an example showing that the invariant
 in the group
$(\Z^N/(\id-A)\Z^N \otimes
\Z^N/(\id-A^t)\Z^N, 
e_A)
$
is strictly finer than 
the K-group
$K_0(\OA)$.
 We note that Enomoto--Fujii--Watatani in \cite{EFW1981} 
listed a complete classification table of
Cuntz--Krieger algebras $\OA$ 
in terms of its K-groups for which its sizes of matrices are three.

Let $A =
\begin{bmatrix}
1 & 1& 1 \\
1 & 1& 1 \\
1 & 1& 1 
\end{bmatrix}.$
Since
$(\id - A) 
\begin{bmatrix}
l\\
m\\
n
\end{bmatrix}
=
\begin{bmatrix}
-m-n\\
-l-n\\
-l-m
\end{bmatrix},
$
the map
$$
\varphi:
\begin{bmatrix}
a\\
b\\
c
\end{bmatrix}
\in \Z^3
\longrightarrow
[a + b+ c] \in \Z/2\Z
$$
 induces an isomorphism
$\bar{\varphi}: \Z^3/(\id - A)\Z^3 \longrightarrow \Z/2\Z.$
Hence we have an isomorphism
\begin{equation*}
\tilde{\varphi}:=\bar{\varphi}\otimes\bar{\varphi}:
\Z^3/(\id - A)\Z^3 \otimes \Z^3/(\id - A^t)\Z^3 \longrightarrow
 \Z/2\Z\otimes\Z/2\Z \cong\Z/2\Z.
\end{equation*}
Since 
$\tilde{\varphi}(e_i\otimes e_i) 
=\bar{\varphi}(e_i) \otimes\bar{\varphi}(e_i)= 1\otimes 1,$
we then have
\begin{equation*}
\tilde{\varphi}(e_A) 
= [1\otimes 1] +[1\otimes 1] +[1\otimes 1] 
= [1] \quad \text{ in }\Z/2\Z 
\end{equation*}
so that  
$$
(\Z^3/(\id - A)\Z^3 \otimes \Z^3/(\id - A^t)\Z^3, e_A) \cong (\Z/2\Z, [1]).
$$

On the other hand,
let
$B =
\begin{bmatrix}
1 & 1& 1 \\
1 & 1& 0 \\
1 & 1& 0 
\end{bmatrix}$
and hence
$B^t =
\begin{bmatrix}
1 & 1& 1 \\
1 & 1& 1 \\
1 & 0& 0 
\end{bmatrix}.$
Since
$$
(\id - B) 
\begin{bmatrix}
l\\
m\\
n
\end{bmatrix}
=
\begin{bmatrix}
-m-n\\
-l\\
-l-m+n
\end{bmatrix},
\qquad
(\id - B^t) 
\begin{bmatrix}
l\\
m\\
n
\end{bmatrix}
=
\begin{bmatrix}
-m-n\\
-l-n\\
-l+n
\end{bmatrix},
$$
the maps
$$
\psi:
\begin{bmatrix}
a\\
b\\
c
\end{bmatrix}
\in \Z^3
\longrightarrow
[a + b+ c] \in \Z/2\Z,
\qquad
\psi^t:
\begin{bmatrix}
a\\
b\\
c
\end{bmatrix}
\in \Z^3
\longrightarrow
[b+ c] \in \Z/2\Z
$$
satisfy
$$
\psi(
(\id - B)
\begin{bmatrix}
l\\
m\\
n
\end{bmatrix}
)
=2(-l-m),
\qquad 
\psi^t(
(\id - B^t)
\begin{bmatrix}
l\\
m\\
n
\end{bmatrix}
)
= -2 l
$$
so that they induce isomorphisms
$$
\bar{\psi}: \Z^3/(\id - B)\Z^3 \longrightarrow \Z/2\Z,\qquad
 \bar{\psi}^t: \Z^3/(\id - B^t)\Z^3 \longrightarrow \Z/2\Z
$$
and
\begin{equation*}
\tilde{\psi}:=\bar{\psi}\otimes\bar{\psi}^t:
\Z^3/(\id - B)\Z^3 \otimes \Z^3/(\id - B^t)\Z^3 
\longrightarrow
 \Z/2\Z\otimes\Z/2\Z \cong\Z/2\Z.
\end{equation*}
Since 
$$
\tilde{\psi}(e_i\otimes e_i)
= 
\bar{\psi}(e_i)\otimes\bar{\psi}^t(e_i)
=
\begin{cases}
[ 1\otimes 0] =[0] & \text{ if } i=1,\\
[ 1\otimes 1 ] =[1] & \text{ if } i=2,3,
\end{cases}
$$
we then have
\begin{equation*}
\tilde{\psi}(e_A) 
= [1\otimes 0] +[1\otimes 1] +[1\otimes 1] = [0] \quad \text{ in }\Z/2\Z 
\end{equation*}
so that  
$$
(\Z^3/(\id - B)\Z^3 \otimes \Z^3/(\id - B^t)\Z^3, e_B) \cong (\Z/2\Z, [0]).
$$
\begin{proposition}\label{prop:example}
Let
$A =
\begin{bmatrix}
1 & 1& 1 \\
1 & 1& 1 \\
1 & 1& 1 
\end{bmatrix}$
and
$B =
\begin{bmatrix}
1 & 1& 1 \\
1 & 1& 0 \\
1 & 1& 0  
\end{bmatrix}.$
They satisfy
$K_0(\OA) \cong K_0(\OB)(\cong \Z/2\Z) $
and 
$\det(\id -A) = \det(\id -B) (= -2)$.
However 
\begin{gather*}
(\Z^3/(\id - A)\Z^3 \otimes \Z^3/(\id - A^t)\Z^3, e_A) \cong (\Z/2\Z, [1]), \\
(\Z^3/(\id - B)\Z^3 \otimes \Z^3/(\id - B^t)\Z^3, e_B) \cong (\Z/2\Z, [0]).
\end{gather*}
\end{proposition}

In the rest of this section, we will deal with square matrices 
with entries in nonnegative integers.
Such matrices are called nonnegative integral matrices.
A nonnegative integral matrix is said to be essential 
if none of its rows or columns is zero vector.
Let $A =[A(i,j)]_{i,j=1}^N$
be an $N\times N$ essential nonnegative integral matrix.
The matrix defines a finite directed graph $G_A=(V_A, E_A)$ with $N$ vertices
$V_A=\{v_1,\dots,v_N\}$
and $A(i,j)$ directed edges from the vertex $v_i$ to the vertex $v_j$
for $i,j=1,\dots,N$.
The directed edges are denoted by
$\{a_1,\dots,a_{N_A}\} =E_A$.
For an edge $a_k\in E_A$, denote by 
$s(a_k), t(a_k)$ its source vertex, terminal vertex, respectively.
The directed graph $G_A$ has the $N_A\times N_A$ transition matrix
$A^G=[A^G(i,j)]_{i,j=1}^{N_A}$ of edges defined by
\begin{equation*}
A^G(i,j) =
\begin{cases}
1 & \text{ if } t(a_i) = s(a_j),\\
0 & \text{ otherwise, }
\end{cases}
\qquad
i, j=1,\dots,N_A.
\end{equation*}
As in \cite[Remark 2.16]{CK} and \cite[Section 4]{Ro},
the Cuntz--Krieger algebra $\OA$
for the nonnegative integral matrix $A$
is defined to be the Cuntz--Krieger algebra 
$\mathcal{O}_{A^G}$
for the matrix $A^G$ with entries in $\{0,1\}$.
It is well-known that there exist rectangular nonnegative integral
matrices $R, S$ such that $A = RS, A^G = SR$ (cf. \cite{LM}).
As in \cite[Lemma 4.5]{MaDocMath2017},
the left multiplication of the matrix $S^t$
induces an isomorphism
$m_{S^t}: \Z^{N_A}/(\id-(A^G)^t)\Z^{N_A} 
\longrightarrow
\Z^N/(\id-A^t)\Z^N$
such that 
$m_{S^t}([1_{N_A}]) = [1_N]$,
where 
$ 1_{N_A} =[1,\dots,1] \in \Z^{N_A}, 
  1_N =[1,\dots,1]\in \Z^N$.
Let $1_{\OA}$ be the unit of the Cuntz--Krieger algebra $\OA$.
By \cite[Proposition 3.1]{Cu3},
there exists an isomorphism from $K_0(\mathcal{O}_{A^G})$ to
$\Z^{N_A}/(\id-(A^G)^t)\Z^{N_A} $
that sends the class $[1_{\OA}]$ of 
$1_{\OA}$ to the class  $[1_{N_A}]$ of $1_{N_A}$.
Hence for a nonnegative integral matrix $A$,
there exists an isomorphism from
$K_0(\OA)$ to  $\Z^N/(\id-A^t)\Z^N$
that sends the class of the unit $[1_{\OA}]$ 
of $\OA$ to the class $[1_N]$ of $ 1_N$.
We define the element
$[e_A]$ in the group
$\Z^N/(\id-A)\Z^N \otimes
\Z^N/(\id-A^t)\Z^N
$
by the same formula 
\eqref{eq:eA}
as that for  matrices with entries in $\{0,1\}$.
We notice the following lemma.
\begin{lemma}
There exists an isomorphism $\Phi$ of groups
from 
$\Z^{N_A}/(\id-A^G)\Z^{N_A} \otimes
\Z^{N_A}/(\id-(A^G)^t)\Z^{N_A}
$
onto
$\Z^N/(\id-A)\Z^N \otimes
\Z^N/(\id-A^t)\Z^N
$
such that 
$\Phi(e_{A^G}) = e_A$. 
\end{lemma} 
\begin{proof}
Let $R,S$ be  rectangular nonnegative integral
matrices $R, S$ satisfying $A = RS, A^G = SR$.
As in the proof of Theorem \ref{thm:Kinvariant},
the isomorphism
$
m_R\otimes m_{S^t}: \Z^{N_A}/(\id-A^G)\Z^{N_A} \otimes
\Z^{N_A}/(\id-(A^G)^t)\Z^{N_A}
\longrightarrow\Z^N/(\id-A)\Z^N \otimes
\Z^N/(\id-A^t)\Z^N
$
satisfies
$m_R\otimes m_{S^t}(e_{A^G}) = e_A$.
\end{proof}
We may obtain the following proposition
in a similar way to the proof of Proposition \ref{prop:SEinvariant}.
 \begin{proposition}\label{prop:NNSEinvariant}
Let $A =[A(i,j)]_{i,j=1}^N$
be an $N\times N$ essential nonnegative integral matrix.
The pair
$
(\Z^N/(\id-A)\Z^N \otimes
\Z^N/(\id-A^t)\Z^N, e_A)
$
is invariant under shift equivalence.
\end{proposition}

 We will  present an example of nonnegative integral matrix $A$
such that 
the two $C^*$-algebras
$\WRA$ and $\OTA\otimes\OA$ are not isomorphic.

Let $A =
\begin{bmatrix}
4 & 1 \\
1 & 0 
\end{bmatrix}.$
Since
$(\id - A) 
\begin{bmatrix}
l\\
m
\end{bmatrix}
=
\begin{bmatrix}
-3l -m\\
-l+m
\end{bmatrix},
$
the map
$
\varphi:
\begin{bmatrix}
l\\
m
\end{bmatrix}
\in \Z^2
\longrightarrow
[l +m] \in \Z/4\Z
$
 induces  isomorphisms
$\bar{\varphi}: \Z^2/(\id - A)\Z^2 \longrightarrow \Z/4\Z$
and
\begin{equation*}
\tilde{\varphi}:=\bar{\varphi}\otimes\bar{\varphi}:
\Z^2/(\id - A)\Z^2 \otimes \Z^2/(\id - A^t)\Z^2 \longrightarrow
 \Z/4\Z\otimes\Z/4\Z \cong\Z/4\Z.
\end{equation*}
Since $\tilde{\varphi}(e_i\otimes e_i) 
=\bar{\varphi}(e_i) \otimes\bar{\varphi}(e_i)= 1\otimes 1,$
we have
\begin{equation*}
\tilde{\varphi}(e_A) 
= [1\otimes 1] +[1\otimes 1] 
= [2] \quad \text{ in }\Z/4\Z. 
\end{equation*}
On the other hand,
we have
$\tilde{\varphi}([1_2]\otimes [1_2]) 
=\varphi(
\begin{bmatrix}
1\\
1
\end{bmatrix}
)
\otimes
\varphi^t(
\begin{bmatrix}
1\\
1
\end{bmatrix}
) 
= [2\otimes 2]
=[0]$
in $\Z/4\Z$.
We thus have
\begin{gather*}
(\Z^2/(\id - A)\Z^2 \otimes \Z^2/(\id - A^t)\Z^2, e_A) \cong (\Z/4\Z, [2]), \\
(\Z^2/(\id - A)\Z^2 \otimes \Z^2/(\id - A^t)\Z^2, [1_2]\otimes[1_2]) \cong (\Z/4\Z, [0]),
\end{gather*}
so that 
the algebras
$\WRA$ and $\OTA\otimes\OA$ are not isomorphic
by classification theorem of 
unital, purely infinite, simple nuclear $C^*$-algebras 
(\cite{Kirchberg}, \cite{Phillips}).

 
\section{KMS states on $\WRA$}
In this section, we will study KMS states on the $C^*$-algebra $\WRA$
for the diagonal action $\delta^A$. 
Following after \cite{BR}, 
 we will define KMS states in the following way.
For a one-parameter automorphism group $\alpha_t, t \in {\mathbb{R}}$
on a $C^*$-algebra $\mathcal{A}$ and a real number $\gamma \in {\mathbb{R}},$
 a state $\psi$ on $\mathcal{A}$ is called a KMS state for the action $\alpha$
if $\psi$ satisfies 
\begin{equation}
\psi(X \alpha_{i\gamma}(Y)) = \psi(Y X) \label{eq:KMSstate}
\end{equation}
for all $X, Y$ in a norm dense $\alpha$-invariant $*$-subalgebra of the set of entire analytic elements for $\alpha$ in $\mathcal{A}$.
The value $\gamma$ is called the inverse temperature and
the condition \eqref{eq:KMSstate} is called the KMS condition.

Let $\beta$ 
be the Perron--Frobenius eigenvalue for an irreducible matrix $A$
with entries in $\{0,1\}$.
It has been shown in \cite{EFW1984} that 
KMS states for gauge action on Cuntz--Krieger algebra $\OA$
exists if and only if its inverse temperature is $\log \beta$,
and the admitted KMS state is unique.
Let us denote by
$\varphi$ the unique KMS state for gauge action on $\OA$.  
Similarly we denote by 
$\varphi^t$ the unique KMS state for gauge action on $\OTA$
As in \cite{EFW1984}, the vector
$\begin{bmatrix}
\varphi(S_1S_1^*)\\
\vdots \\
\varphi(S_NS_N^*)
\end{bmatrix}
$
gives rise to the unique normalized positive eigenvector 
of $A$ for the eigenvalue $\beta$.
Hence we have
\begin{equation*}
\beta
\begin{bmatrix}
\varphi(S_1S_1^*)\\
\vdots \\
\varphi(S_NS_N^*)
\end{bmatrix}
=
\begin{bmatrix}
A(1,1)&\cdots&A(1,N)\\
\vdots&        &\vdots\\
A(N,1)&\cdots&A(N,N)
\end{bmatrix}
\begin{bmatrix}
\varphi(S_1S_1^*)\\
\vdots \\
\varphi(S_NS_N^*)
\end{bmatrix}
=
\begin{bmatrix}
\varphi(S_1^*S_1)\\
\vdots \\
\varphi(S_N^*S_N)
\end{bmatrix}
\end{equation*}
so that 
$\beta\varphi(S_i S_i^*) = \varphi(S_i^* S_i), i=1,\dots,N$
and more generally
\begin{equation*}
\beta^m\varphi(S_{\mu_1 \cdots \mu_m}S_{\mu_1 \cdots \mu_m}^*) 
= \varphi(S_{\mu_1 \cdots \mu_m}^*S_{\mu_1 \cdots \mu_m}),
\quad (\mu_1,\dots,\mu_m) \in B_m(\bar{X}_A).
\end{equation*}
Therefore we have
\begin{equation}
\varphi(S_{\mu_m} S_{\mu_m}^*) 
=\frac{1}{\beta}\varphi(S_{\mu_m}^* S_{\mu_m})
=\frac{1}{\beta}\varphi(S_{\mu_1\cdots\mu_m}^* S_{\mu_1\cdots\mu_m})
=\beta^{m-1}\varphi(S_{\mu_1\cdots\mu_m} S_{\mu_1\cdots\mu_m}^*)\label{eq:varphis}
\end{equation}
and similarly
\begin{equation}
\varphi^t(T_{\xi_1}T_{\xi_1}^*) 
=\beta^{k-1}\varphi^t(T_{\xi_k\cdots\xi_1}T_{\xi_k\cdots\xi_1}^*),
\qquad (\xi_k,\dots,\xi_1)\in B_k(\bar{X}_{A^t}).
\label{eq:varphit}
\end{equation}
Let
$[a_i]_{i=1}^N
$
and
$[b_i]_{i=1}^N
$
be the positive eigenvectors of $A$ and $A^t$ for the eigenvalue $\beta$,
respectively
satisfying 
\begin{equation*}
\sum_{i=1}^N a_i b_i =1.
\end{equation*}
For  admissible words
$\xi=(\xi_1,\dots,\xi_k)\in B_k(\bar{X}_A)$
and
$\nu=(\nu_1,\dots,\nu_n) \in B_n(\bar{X}_A)$,
put
$\xi\nu =(\xi_1, \dots, \xi_k,\nu_1,\dots,\nu_n) \in B_{k+n}(\bar{X}_A)$.
For $i\in \Z$,
let us denote by $U_{[\xi\nu]_i^{i+k+n-1}}$ the cylinder set of $\bar{X}_A$
such that 
 \begin{equation*}
U_{[\xi\nu]_i^{i+k+n-1}} 
= \{ (x_j)_{j\in \Z}\in \bar{X}_A \mid
x_i  =\xi_1, \dots, x_{i+k-1} =\xi_k, 
x_{i+k} = \nu_1,\dots, x_{i+k+n-1} =\nu_n\}. 
\end{equation*}
 In  \cite{Parry}, W. Parry proved that there exists a unique invariant measure $\mu$ on 
$\bar{X}_A$ 
of maximal entropy. 
It is called the Parry measure, which satisfies the following equality
\begin{equation}
\mu(U_{[\xi\nu]_i^{i+k+n-1}}) = b_{\xi_1}a_{\nu_n}\beta^{-(k+n-1)}, \quad i \in \Z.
\label{eq:Parry}
\end{equation} 
Let
$C^*(G_A^a)$ be the groupoid $C^*$-algebra
 for the groupoid $G_A^a$.
As in the Putnam's paper \cite{Putnam1} and his lecture note \cite{Putnam2},
the algebra is an AF-algebra with a tracial state
$\Tr$ defined by
\begin{equation*}
\Tr(f) = \int_{\bar{X}_A} f(x,x) d\mu(x) \qquad 
\text{ for } f \in C_c(G_A).
\end{equation*}
 Let us define a state $\tilde{\varphi}$ on $\WRA$ by setting
\begin{equation*}
\tilde{\varphi} 
=\frac{1}{\sum_{j=1}^N\varphi^t(T_jT_j^*)\varphi(S_j^*S_j)}\varphi^t\otimes\varphi
\quad
\text{ on }
\WRA =E_A(\OTA\otimes\OA)E_A.
\end{equation*}
Since
$(\varphi^t\otimes\varphi)(E_A) =
\sum_{j=1}^N\varphi^t(T_jT_j^*)\varphi(S_j^*S_j),
$
we know that $\tilde{\varphi}$
gives rise to a state on $\WRA$. 
We know more about $\tilde{\varphi}$ in the following way.
\begin{proposition}\label{prop:kms}
\begin{enumerate}
\renewcommand{\theenumi}{\roman{enumi}}
\renewcommand{\labelenumi}{\textup{(\theenumi)}}
\item
The state $\tilde{\varphi}$ is a KMS state on $\WRA$ for the diagonal action $\delta^A$
at the inverse temperature $\log\beta$.
\item
The restriction of  $\tilde{\varphi}$ to the subalgebra $C(\bar{X}_A)$ coincides with the Parry measure $\mu$ on $\bar{X}_A$.
\item The formula
\begin{equation}
\tilde{\varphi}(Y) = \Tr\left(\iint_{\T^2} \gamma_{r,s}^A(Y)dr ds\right)  \qquad
\text{ for } Y \in \WRA \label{eq:trace}
\end{equation}
holds.
\end{enumerate}
\end{proposition}
\begin{proof}
(i) 
For
\begin{gather*}
\mu =(\mu_1,\dots,\mu_m), \, \nu =(\nu_1,\dots,\nu_n), \,
\mu' =(\mu'_1,\dots,\mu'_{m'}), \, \nu' =(\nu'_1,\dots,\nu'_{n'})
 \in B_*(\bar{X}_A),\\
\bar{\xi} =(\xi_k,\dots,\xi_1), \, \bar{\eta} =(\eta_l,\dots,\eta_1), \,
\bar{\xi}' =(\xi'_{k'},\dots,\xi'_1), \, \bar{\eta}' =(\eta'_{l'},\dots,\eta'_1)
 \in B_*(\bar{X}_{A^t})
\end{gather*}
with
$A(\xi_k,\mu_1) = A(\eta_l,\nu_1) =
A(\xi'_{k'}, \mu'_1) = A(\eta'_{l'},\nu'_1) =1,
$
put
$$
x =T_{\bar{\xi}}T_{\bar{\eta}}^*\otimes S_\mu S_\nu^*,\qquad
x' =T_{\bar{\xi}'}T_{\bar{\eta}'}^*\otimes S_{\mu'} S_{\nu'}^*
 \in \WRA. 
$$
It then follows that 
\begin{align*}
& (\varphi^t\otimes\varphi)(E_A)\cdot \tilde{\varphi}(x'\delta^A_{i\log\beta}(x)) \\
=&(\varphi^t\otimes\varphi)(
(T_{\bar{\xi}'}T_{\bar{\eta}'}^*\otimes S_{\mu'} S_{\nu'}^*)
(\alpha_{i\log\beta}^{A^t}(T_{\bar{\xi}}T_{\bar{\eta}}^*)
\otimes \alpha_{i\log\beta}^A(S_\mu S_\nu^*))) \\
=&\varphi^t(
T_{\bar{\xi}'}T_{\bar{\eta}'}^*\alpha_{i\log\beta}^{A^t}(T_{\bar{\xi}}T_{\bar{\eta}}^*))
\varphi(
S_{\mu'} S_{\nu'}^*\alpha_{i\log\beta}^A(S_\mu S_\nu^*)) \\
=& \varphi^t(
T_{\bar{\xi}}T_{\bar{\eta}}^*T_{\bar{\xi}'}T_{\bar{\eta}'}^*)
\varphi(S_\mu S_\nu^*S_{\mu'} S_{\nu'}^*) \\
=&(\varphi^t\otimes\varphi)(
(T_{\bar{\xi}}T_{\bar{\eta}}^*\otimes S_\mu S_\nu^*)
(T_{\bar{\xi}'}T_{\bar{\eta}'}^*\otimes S_{\mu'} S_{\nu'}^*))\\
=&(\varphi^t\otimes\varphi)(xx') \\
=
& (\varphi^t\otimes\varphi)(E_A)\cdot \tilde{\varphi}(xx'),
\end{align*}
thus proving that 
$\tilde{\varphi}$ is a KMS state $\WRA$ for the diagonal action $\delta^A$
at the inverse temperature $\log\beta$.

(ii)
Put 
$$
\bar{a}_i = \frac{a_i}{\sum_{i=1}^N a_i}  = \varphi(S_i S_i^*),
\qquad
\bar{b}_i = \frac{b_i}{\sum_{i=1}^N b_i}  = \varphi^t(T_i T_i^*)
$$
so that 
\begin{equation*}
\sum_{i=1}^N\varphi(S_i S_i^*)\varphi^t(T_i T_i^*)
 = \frac{1}{(\sum_{i=1}^N a_i)\cdot(\sum_{i=1}^N b_i)}.
\end{equation*}
It then follows that 
\begin{align*}
\mu(U_{[\xi\nu]_i^{i+k+n-1}})
=& \bar{b}_{\xi_1} \cdot ({\sum_{i=1}^N b_i})\cdot
     \bar{a}_{\nu_n}(\sum_{i=1}^N a_i) \cdot \beta^{-(k+n-1)} \\
=& \varphi^t(T_{\xi_1} T_{\xi_1}^*) \cdot ({\sum_{i=1}^N b_i})\cdot
     \varphi(S_{\nu_n} S_{\nu_n}^*)(\sum_{i=1}^N a_i) \cdot \beta^{-(k+n-1)} \\
=&\frac{1}{\sum_{i=1}^N\varphi^t(T_i T_i^*)\varphi(S_i S_i^*)}\cdot
\varphi^t(T_{\xi_1} T_{\xi_1}^*) \varphi(S_{\nu_n} S_{\nu_n}^*) \cdot \beta^{-(k+n-1)}. 
\end{align*}
By using \eqref{eq:varphis} and \eqref{eq:varphit}
we thus have
\begin{align*}
\mu(U_{[\xi\nu]_i^{i+k+n-1}})
=&\frac{1}{\sum_{i=1}^N\varphi^t(T_i T_i^*)\varphi(S_i S_i^*)}\cdot
\frac{1}{\beta}\cdot
\varphi^t(T_{\xi_k\cdots\xi_1} T_{\xi_k\cdots\xi_1}^*)
 \varphi(S_{\nu_1\cdots\nu_n} S_{\nu_1\cdots\nu_n}^*) \\
=&\frac{1}{\sum_{i=1}^N\varphi^t(T_i T_i^*)\varphi(S_i^* S_i)}\cdot
\varphi^t(T_{\bar{\xi}} T_{\bar{\xi}}^*)
 \varphi(S_{\nu} S_{\nu}^*) \\ 
=&\frac{1}{(\varphi^t\otimes\varphi)(E_A)}\cdot
\varphi^t(T_{\bar{\xi}} T_{\bar{\xi}}^*)
 \varphi(S_{\nu} S_{\nu}^*) \\
=&\tilde{\varphi}(T_{\bar{\xi}} T_{\bar{\xi}}^*\otimes S_{\nu} S_{\nu}^*).
\end{align*}

(iii)
For 
\begin{gather*}
 \mu =(\mu_1,\dots,\mu_m), \quad \nu =(\nu_1,\dots,\nu_n) \in B_*(\bar{X}_A),\\
 \bar{\xi} =(\xi_k,\dots,\xi_1), \quad \bar{\eta} =(\eta_l,\dots,\eta_1) \in B_*(\bar{X}_{A^t}),
\end{gather*}
satisfying
$ 
 A(\xi_k,\mu_1) = A(\eta_l,\nu_1) =1,
$
 it is direct to see the following equalities
\begin{align*}
 \tilde{\varphi}(T_{\bar{\xi}}T_{\bar{\eta}}^*\otimes S_\mu S_\nu^*) 
=&\varphi^t(T_{\bar{\xi}}T_{\bar{\eta}}^* )\varphi(S_\mu S_\nu^*) \\
=&
{\begin{cases}
\varphi^t(T_{\bar{\xi}}T_{\bar{\xi}}^*) \varphi(S_\nu S_\nu^*) 
& \text{ if }
 \bar{\xi}=\bar{\eta}, \mu = \nu, \\
0 & \text{ otherwise } 
\end{cases} }\\
=& 
{\begin{cases}
\mu(U_{[\xi\nu]_i^{i+k+n-1}})
& \text{ if }
 \bar{\xi}=\bar{\eta}, \mu = \nu, \\
0 & \text{ otherwise. } 
\end{cases} }
\end{align*}
Since the above value coincides with
$$
 \Tr\left(
\iint_{\T^2} \gamma_{r,s}^A(T_{\bar{\xi}}T_{\bar{\eta}}^*\otimes S_\mu S_\nu^*)dr ds
\right), 
$$
we know the formula \eqref{eq:trace}.
\end{proof}

\medskip

We finally prove that  a KMS state on $\WRA$ for the diagonal action
$\delta^A$ exists only if  at  the inverse temperature $\log\beta$.
We will further know that the admitted KMS state is unique. 
In order to avoid non essential difficulty, 
we assume that the irreducible matrix $A$ with entries in $\{0,1\}$
is aperiodic so that 
there exists $n_0\in \N$ such that 
$A^{n_0}(i,j)\ge 1$ for all $i,j = 1,\dots,N$.
Let $\psi$ be a KMS state on $\WRA$ for the diagonal action
$\delta^A$ at the inverse temperature 
$\log\gamma$ for $1<\gamma\in {\mathbb{R}}.$
We will prove that $\gamma = \beta$: the Perron--Frobenius eigenvalue of $A$
and $\psi = \tilde{\varphi}.$

For $i,j \in \{1,\dots,N\}$ and
$\mu =(\mu_1,\dots,\mu_m),\, 
 \nu =(\nu_1,\dots,\nu_n)$ such that 
$(i,\mu_1,\dots,\mu_m, j) \in B_{m+2}(\bar{X}_A),\,
  (i,\nu_1,\dots,\nu_n, j) \in B_{n+2}(\bar{X}_A)$,  
we set a partial isometry
\begin{equation}
V_{\nu, \mu}(i,j) = T_i^*T_{\nu_1}^*\cdots T_{\nu_n}^*T_j^*\otimes
S_i S_{\mu_1}\cdots S_{\mu_m} S_j \label{eq:Vnmij}
\end{equation}
Since $T_i^*\otimes S_i, T_j^*\otimes S_j \in \WRA$, we know that 
$V_{\nu,\mu}(i,j)$ belongs to $\WRA$.
We then have the identities
\begin{align*}
V_{\nu, \mu}(i,j)V_{\nu, \mu}(i,j)^*
= & T_i^*T_{\nu_1}^*\cdots T_{\nu_n}^*T_j^*
T_j T_{\nu_n}\cdots T_{\nu_1}T_i
\otimes
S_i S_{\mu_1}\cdots S_{\mu_m} S_j
S_j^*S_{\mu_m}^*\cdots S_{\mu_1}^*S_i^* \\
=& T_i^*T_i
\otimes
S_i S_{\mu_1}\cdots S_{\mu_m} S_j
S_j^*S_{\mu_m}^*\cdots S_{\mu_1}^*S_i^*\\
\intertext{and}
V_{\nu, \mu}(i,j)^*V_{\nu, \mu}(i,j)
= & T_j T_{\nu_n}\cdots T_{\nu_1}T_i T_i^*T_{\nu_1}^*\cdots T_{\nu_n}^*T_j^*
\otimes
S_j^*S_{\mu_m}^*\cdots S_{\mu_1}^*S_i^*
S_i S_{\mu_1}\cdots S_{\mu_m} S_j
 \\
=& T_j T_{\nu_n}\cdots T_{\nu_1}T_i T_i^*T_{\nu_1}^*\cdots T_{\nu_n}^*T_j^*
\otimes
S_j^* S_j
\end{align*}
For $p \in \Z$, denote by 
$\WRA^{\delta^A}(p)$ the $p$th spectral subspace of $\WRA$ for the action $\delta^A$.
\begin{lemma}\label{lem:6.2}
Suppose that $X\in \WRA$ 
belongs to $\WRA^{\delta^A}(p)$ for some $p\ne 0$.
Then we have
$\psi(X) =0$
\end{lemma}
\begin{proof}
We may assume $p>0$. 
For $i,j =1,\dots,N$, 
let
$ 
\mu =(\mu_1,\dots,\mu_{n_0 +p})
$
be  an admissible  word 
such that 
$(i,\mu_1,\dots,\mu_{n_0 +p},j) \in B_{n_0 + p +2}(\bar{X}_A)$.
Take $ 
\nu =(\nu_1,\dots,\nu_{n_0})$ with
$(i,\nu_1,\dots,\nu_{n_0},j) \in B_{n_0}(\bar{X}_A)$
and consider the partial isometry
\begin{equation*}
V_{\nu, \mu}(i,j) = T_i^*T_{\nu_1}^*\cdots T_{\nu_{n_0}}^*T_j^*
\otimes
S_i S_{\mu_1}\cdots S_{\mu_{n_0 +p}} S_j.
\end{equation*}
The partial isometry $V_{\nu, \mu}(i,j) $ belongs to 
$\WRA^{\delta^A}(p)$ and satisfies
\begin{equation*}
V_{\nu, \mu}(i,j)V_{\nu, \mu}(i,j)^*
=  T_i^*T_i
\otimes
S_i S_{\mu_1}\cdots S_{\mu_{n_0 +p}} S_j
S_j^*S_{\mu_{n_0 +p}}^*\cdots S_{\mu_1}^*S_i^*.
\end{equation*}
We then have
\begin{equation*}
E_A = 
\sum_{i,j=1}^N \sum_{\mu \in B_{n_0 +p}(\bar{X}_A)}
V_{\nu, \mu}(i,j)V_{\nu, \mu}(i,j)^*.
\end{equation*}
It then follows that 
\begin{align*}
\psi(X) 
= & \psi(E_A X) \\
= & \sum_{i,j=1}^N \sum_{\mu \in B_{n_0 +p}(\bar{X}_A)}
\psi(V_{\nu, \mu}(i,j)V_{\nu, \mu}(i,j)^* X) \\
= & \sum_{i,j=1}^N \sum_{\mu \in B_{n_0 +p}(\bar{X}_A)}
\psi(V_{\nu, \mu}(i,j)^* X \delta^A_{i\log\gamma}(V_{\nu, \mu}(i,j))) \\
= & \frac{1}{\gamma^p} \sum_{i,j=1}^N \sum_{\mu \in B_{n_0 +p}(\bar{X}_A)}
\psi(V_{\nu, \mu}(i,j)^* X V_{\nu, \mu}(i,j)) \\
= & \frac{1}{\gamma^p} \sum_{i,j=1}^N \sum_{\mu \in B_{n_0 +p}(\bar{X}_A)}
\psi( V_{\nu, \mu}(i,j) \delta^A_{i\log\gamma}(V_{\nu, \mu}(i,j)^* X)) \\
= & \frac{1}{\gamma^p} \sum_{i,j=1}^N \sum_{\mu \in B_{n_0 +p}(\bar{X}_A)}
\psi( V_{\nu, \mu}(i,j) V_{\nu, \mu}(i,j)^* X) \\
=&\frac{1}{\gamma^p}\psi(E_A X).
\end{align*}
Since $\gamma>1$, we have $\psi(X)=0$. 
\end{proof}
Since 
$\RA$ is the fixed point algebra $(\WRA)^{\delta^A}$ of $\WRA$
under $\delta^A$, we may define a conditional expectation 
${\mathcal{E}}_A: \WRA \longrightarrow \RA$
 by
\begin{equation}
\mathcal{E}_A(X) = \int_{\mathbb{T}}\delta^A_t(X) dt, \qquad X \in \WRA.
\end{equation}
The preceding lemma implies the following lemma.
\begin{lemma}\label{lem:6.3}
Let $\psi_0$ be the restriction of $\psi$ to the subalgebra $(\WRA)^{\delta^A}$.
Then $\psi_0$ is a tracial state on $\RA$ such that 
$\psi = \psi_0 \circ {\mathcal{E}}_A.$
\end{lemma}
Hence the value of KMS state is determined on the subalgebra $\RA$. 
Recall that  $U_A$ denotes the unitary
 $U_A = \sum_{i=1}^N T_i^* \otimes S_i$ which belongs to $\RA$.
\begin{lemma}\label{lem:6.4}
$\psi(U_AXU_A^*)=\psi(X)$ for all $X \in \WRA$.
\end{lemma} 
\begin{proof}
Since $U_A$ is fixed under the action $\delta^A$, we have 
\begin{equation*} 
\psi(X) = \psi(U_A^* U_A X) 
          = \psi(U_A X\delta^A_{i\log\gamma}(U_A^*)) 
          = \psi(U_A X U_A^*).
\end{equation*}
\end{proof}
As in \cite[Proposition 9.9]{MaPre2017}, the automorphism
$\Ad(U_A)$ behaves like the  shift on $\WRA$.
Lemma \ref{lem:6.4} tells us that the KMS state is invariant nuder the shift.
The following lemma is crucial in our discussions. 
\begin{lemma}\label{lem:6.5}
Let  
$X =T_{\bar{\xi}}T_{\bar{\eta}}^*\otimes S_\mu S_\nu^* \in \RA 
$
where
\begin{gather*}
 \mu =(\mu_1,\dots,\mu_m), \nu =(\nu_1,\dots,\nu_n) \in B_*(\bar{X}_A),\\
 \bar{\xi} =(\xi_k,\dots,\xi_1), \bar{\eta} =(\eta_l,\dots,\eta_1) \in B_*(\bar{X}_{A^t}).
\end{gather*}
Suppose that $\psi(X) \ne 0$.
Then  we have
$k=l, m=n$ and $\mu = \nu, \bar{\xi} = \bar{\eta}$.
\end{lemma}
\begin{proof}
Since
$X$ belongs to $\RA$, we have 
 $ A(\xi_k,\mu_1) = A(\eta_l,\nu_1) =1$ and $k-l = n-m$.
We may assume that $k\ge l$ and hence $n\ge m$.
It then follows that 
\begin{align*}
\psi(X)
=& \psi(T_{\xi_k}\cdots T_{\xi_1}T_{\eta_1}^*\cdots T_{\eta_l}^*
\otimes S_{\mu_1}\cdots S_{\mu_m} S_{\nu_n}^*\cdots S_{\nu_1}^* )\\
=& \psi((T_{\xi_k}\cdots T_{\xi_1}T_{\eta_1}^*\cdots T_{\eta_l}^*
\otimes S_{\mu_1}\cdots S_{\mu_m} S_{\nu_n}^*\cdots S_{\nu_1}^*)
\cdot (T_{\eta_l}T_{\eta_l}^*\otimes S_{\nu_1}S_{\nu_1}^*)
 )\\
=& \psi((T_{\eta_l}T_{\eta_l}^*\otimes S_{\nu_1}S_{\nu_1}^*)\cdot
\delta^A_{i\log\gamma}(T_{\xi_k}\cdots T_{\xi_1}T_{\eta_1}^*\cdots T_{\eta_l}^*
\otimes S_{\mu_1}\cdots S_{\mu_m} S_{\nu_n}^*\cdots S_{\nu_1}^*)
 )\\
=& \frac{1}{\gamma^{k+m-l-n}}
\psi((T_{\eta_l}T_{\eta_l}^*\otimes S_{\nu_1}S_{\nu_1}^*)\cdot
(T_{\xi_k}\cdots T_{\xi_1}T_{\eta_1}^*\cdots T_{\eta_l}^*
\otimes S_{\mu_1}\cdots S_{\mu_m} S_{\nu_n}^*\cdots S_{\nu_1}^*)
 )\\
=& 
\psi(T_{\eta_l}T_{\eta_l}^*\cdot
T_{\xi_k}\cdots T_{\xi_1}T_{\eta_1}^*\cdots T_{\eta_l}^*
\otimes 
S_{\nu_1}S_{\nu_1}^*\cdot S_{\mu_1}\cdots S_{\mu_m} S_{\nu_n}^*\cdots S_{\nu_1}^*).
 \end{align*}
By the assumption $\psi(X) \ne 0$, 
we get $\eta_l = \xi_k$ and $\nu_1 = \mu_1$,
and we have 
\begin{equation*}
\psi(X) =
\psi(T_{\eta_l}
T_{\xi_{k-1}}\cdots T_{\xi_1}T_{\eta_1}^*\cdots T_{\eta_l}^*
\otimes 
 S_{\mu_1}\cdots S_{\mu_m} S_{\nu_n}^*\cdots S_{\nu_1}^*).
 \end{equation*}
As
\begin{align*}
& U_A\cdot T_{\eta_l}
T_{\xi_{k-1}}\cdots T_{\xi_1}T_{\eta_1}^*\cdots T_{\eta_l}^*
\otimes 
 S_{\mu_1}\cdots S_{\mu_m} S_{\nu_n}^*\cdots S_{\nu_1}^*\cdot U_A^*\\
=
& 
T_{\xi_{k-1}}\cdots T_{\xi_1}T_{\eta_1}^*\cdots T_{\eta_{l-1}}^*
\otimes 
  S_{\eta_l}S_{\mu_1}\cdots S_{\mu_m} S_{\nu_n}^*\cdots S_{\nu_1}^*S_{\eta_l}^*,
\end{align*}
Lemma \ref{lem:6.4} shows us the equality
\begin{equation}
\psi(X) =
\psi(
T_{\xi_{k-1}}\cdots T_{\xi_1}T_{\eta_1}^*\cdots T_{\eta_{l-1}}^*
\otimes 
  S_{\eta_l}S_{\mu_1}\cdots S_{\mu_m} S_{\nu_n}^*\cdots S_{\nu_1}^*S_{\eta_l}^*).
\label{eq:6.5.2}
 \end{equation}
We apply the same argument above to the right hand side of \eqref{eq:6.5.2},
and continue these procedures
so that we finally get  
$$
\eta_{l-1} = \xi_{k-1},\, \eta_{l-2} = \xi_{k-2},\, \dots,\,  \eta_{1} = \xi_{k-l+1}
$$
and the identity
\begin{equation*}
\psi(X) =
\psi(
T_{\xi_{k-l}}\cdots T_{\xi_1}T_{\eta_1}^*T_{\eta_1}
\otimes 
S_{\eta_1} S_{\eta_2}\cdots 
  S_{\eta_l}S_{\mu_1}\cdots S_{\mu_m} S_{\nu_n}^*\cdots S_{\nu_1}^*S_{\eta_l}^*
\cdots S_{\eta_2}^*S_{\eta_1}^*).
 \end{equation*}
As $\xi_{k-l+1} =\eta_1$, we see that 
$A(\xi_{k-l},\eta_1) =1$ and hence
$T_{\xi_{k-l}}\cdots T_{\xi_1}T_{\eta_1}^*\otimes 
S_{\eta_1}$
 belongs to the algebra $\WRA$
such that 
\begin{equation*}
\delta^A_{i\log{\gamma}}(
T_{\xi_{k-l}}\cdots T_{\xi_1}T_{\eta_1}^*\otimes S_{\eta_1})
 =
\frac{1}{\gamma^{k-l}}
T_{\xi_{k-l}}\cdots T_{\xi_1}T_{\eta_1}^*\otimes S_{\eta_1}.
\end{equation*}
Hence we have
\begin{align*}
\psi(X) 
=&
\psi(
(T_{\xi_{k-l}}\cdots T_{\xi_1}T_{\eta_1}^* \otimes S_{\eta_1})\cdot (T_{\eta_1}
\otimes 
 S_{\eta_2}\cdots 
  S_{\eta_l}S_{\mu_1}\cdots S_{\mu_m} S_{\nu_n}^*\cdots S_{\nu_1}^*S_{\eta_l}^*
\cdots S_{\eta_2}^*S_{\eta_1}^*)) \\
=&
\psi(
 (T_{\eta_1}\otimes 
S_{\eta_2}\cdots S_{\eta_l}S_{\mu_1}\cdots S_{\mu_m} 
S_{\nu_n}^*\cdots S_{\nu_1}^*S_{\eta_l}^*\cdots S_{\eta_2}^*S_{\eta_1}^*)
\cdot \delta^A_{i\log\gamma}(T_{\xi_{k-l}}\cdots T_{\xi_1}T_{\eta_1}^* \otimes S_{\eta_1})
) \\
=&
\frac{1}{\gamma^{k-l}}
\psi(T_{\eta_1}T_{\xi_{k-l}}\cdots T_{\xi_1}T_{\eta_1}^*
\otimes 
S_{\eta_2}\cdots S_{\eta_l}S_{\mu_1}\cdots S_{\mu_m} 
S_{\nu_n}^*\cdots S_{\nu_1}^*S_{\eta_l}^*\cdots S_{\eta_2}^*S_{\eta_1}^*S_{\eta_1}). 
 \end{align*}
Since 
$S_{\eta_2}^*S_{\eta_1}^*S_{\eta_1} = S_{\eta_2}^*$, we have  
\begin{align*}
\psi(X) 
=&
\frac{1}{\gamma^{k-l}}
\psi(T_{\eta_1}T_{\xi_{k-l}}\cdots T_{\xi_1}T_{\eta_1}^*
\otimes 
S_{\eta_2}\cdots S_{\eta_l}S_{\mu_1}\cdots S_{\mu_m} 
S_{\nu_n}^*\cdots S_{\nu_1}^*S_{\eta_l}^*\cdots S_{\eta_2}^*) \\
=&
\frac{1}{\gamma^{k-l}}
\psi({U_A^*}^{l-1}
 (T_{\eta_1}T_{\xi_{k-l}}\cdots T_{\xi_1}T_{\eta_1}^*
\otimes 
S_{\eta_2}\cdots S_{\eta_l}S_{\mu_1}\cdots S_{\mu_m} 
S_{\nu_n}^*\cdots S_{\nu_1}^*S_{\eta_l}^*\cdots S_{\eta_2}^* U_A^{l-1}) \\
=&
\frac{1}{\gamma^{k-l}}
\psi(
 (T_{\eta_l}\cdots T_{\eta_2}T_{\eta_1}T_{\xi_{k-l}}\cdots T_{\xi_1}T_{\eta_1}^*
T_{\eta_2}^*\cdots T_{\eta_l}^* \otimes 
S_{\mu_1}\cdots S_{\mu_m} S_{\nu_n}^*\cdots S_{\nu_1}^*).  
 \end{align*}
Since $(\eta_l,\dots,\eta_1)= (\xi_k,\dots,\xi_{k-l+1})$, 
we finally obtain that 
\begin{equation*}
\psi(X) 
=
\frac{1}{\gamma^{k-l}}
\psi(X)
\end{equation*}
so that 
$k=l$ and hence $\eta = \xi$.
We similarly see that 
$\mu = \nu$.
\end{proof}
Since any element $X$ of $\RA$ is approximated 
by finite linear combinations of elements of the form
$T_{\bar{\xi}}T_{\bar{\eta}}^*\otimes S_\mu S_\nu^* \in \RA, 
$
we have the following proposition by using Lemma \ref{lem:6.3}.
\begin{proposition}
If an element $X \in \WRA$ satisfies
$\psi(X) \ne 0$,
then $X$ belongs to $C(\bar{X}_A)$.
\end{proposition}
We will next show that the restriction of the KMS state $\psi$
to the commutative subalgebra $C(\bar{X}_A)$
coincides with the state defined by the Parry measure on $\bar{X}_A$.

Recall that the partial isometry
$V_{\nu,\mu}(i,j)$ for
 $i,j \in \{1,\dots,N\}$ and
$\mu =(\mu_1,\dots,\mu_m),\, 
 \nu =(\nu_1,\dots,\nu_n)$ such that 
$(i,\mu_1,\dots,\mu_m, j) \in B_{m+2}(\bar{X}_A),\,
  (i,\nu_1,\dots,\nu_n, j) \in B_{n+2}(\bar{X}_A)$
is defined by 
\eqref{eq:Vnmij}.
We set 
\begin{equation*}
p_{m,\mu}(i,j) =\psi(T_i T_i^*
\otimes
S_{\mu_1}\cdots S_{\mu_m} S_j
S_j^*S_{\mu_m}^*\cdots S_{\mu_1}^*).
\end{equation*}
The following lemma holds.
\begin{lemma}\label{lem 6.7}
\hspace{6cm}
\begin{enumerate}
\renewcommand{\theenumi}{\roman{enumi}}
\renewcommand{\labelenumi}{\textup{(\theenumi)}}
\item
$\psi(V_{\nu, \mu}(i,j)V_{\nu, \mu}(i,j)^*) = p_{m,\mu}(i,j).$
\item
$\psi(V_{\nu, \mu}(i,j)^*V_{\nu, \mu}(i,j)) = p_{n,\nu}(i,j).$
\item
$\psi(V_{\nu, \mu}(i,j)V_{\nu, \mu}(i,j)^*)
=\gamma^{n-m}\psi(V_{\nu, \mu}(i,j)^*V_{\nu, \mu}(i,j)).$
\end{enumerate}
\end{lemma}
\begin{proof}
(i)
Since $T_i^* \otimes S_i$ belongs to $\WRA$ such that 
$\delta^A_{i\log\gamma}(T_i^* \otimes S_i) = T_i^* \otimes S_i,$
we have 
\begin{align*}
\psi(
V_{\nu, \mu}(i,j)V_{\nu, \mu}(i,j)^*)
=& \psi(T_i^*T_i
\otimes
S_i S_{\mu_1}\cdots S_{\mu_m} S_j
S_j^*S_{\mu_m}^*\cdots S_{\mu_1}^*S_i^*)\\
=& \psi((T_i^*\otimes S_i ) \cdot(T_i
\otimes
S_{\mu_1}\cdots S_{\mu_m} S_j
S_j^*S_{\mu_m}^*\cdots S_{\mu_1}^*S_i^*))\\
=& \psi((T_i
\otimes
S_{\mu_1}\cdots S_{\mu_m} S_j
S_j^*S_{\mu_m}^*\cdots S_{\mu_1}^*S_i^*)
\cdot (T_i^*\otimes S_i ) )\\
=& \psi(T_i T_i^*
\otimes
S_{\mu_1}\cdots S_{\mu_m} S_j
S_j^*S_{\mu_m}^*\cdots S_{\mu_1}^*S_i^* S_i  )\\
=& p_{m,\mu}(i,j).
\end{align*}

(ii)
We have 
$
V_{\nu, \mu}(i,j)^*V_{\nu, \mu}(i,j)
= T_j T_{\nu_n}\cdots T_{\nu_1}T_i T_i^*T_{\nu_1}^*\cdots T_{\nu_n}^*T_j^*
\otimes
S_j^* S_j
$
and hence
$$
U_A^{n+1} V_{\nu, \mu}(i,j)^*V_{\nu, \mu}(i,j) U_A^{* n+1}
= T_i T_i^*
\otimes
S_{\nu_1} \cdots S_{\nu_n} S_jS_j^* S_{\nu_n}^*\cdots S_{\nu_1}^*.
$$
By Lemma \ref{lem:6.4},
we have the desired identity.

(iii)
As
$\delta^A_{i\log\gamma}( V_{\nu, \mu}(i,j) ^*)
=\gamma^{n-m} V_{\nu, \mu}(i,j)^*, 
$
the KMS condition for $\psi$ ensures us the desired identity.
\end{proof}
The preceding lemma tells us that the values
$p_{m,\mu}(i,j)$ and $p_{n,\nu}(i,j)$ coincide each other 
for $m=n$ as long as 
$(i,\mu_1,\dots,\mu_m, j) \in B_{m+2}(\bar{X}_A),\,
  (i,\nu_1,\dots,\nu_n, j) \in B_{n+2}(\bar{X}_A).$
Hence the value 
$p_{n,\nu}(i,j)$ does not depend on the choice of the word $\nu$
as long as 
the length of $\nu$ is $n$ and 
$(i,\nu_1,\dots,\nu_n, j) \in B_{n+2}(\bar{X}_A).$
We may thus define 
$p_n(i,j)$ by  $p_{n,\nu}(i,j)$ for some $\nu$ with
$(i,\nu_1,\dots,\nu_n, j) \in B_{n+2}(\bar{X}_A).$
If there are  no word $\nu$ such as 
  $(i,\nu_1,\dots,\nu_n, j) \in B_{n+2}(\bar{X}_A),$
then we define $p_n(i,j)$ to be zero.

\begin{lemma}\label{lem:6.8}
Let $i,j=1,\dots,n$ and $n\in \Zp$.
\begin{enumerate}
\renewcommand{\theenumi}{\roman{enumi}}
\renewcommand{\labelenumi}{\textup{(\theenumi)}}
\item
Assume $A^{n+1}(i,j) >0$ and  $A^{n+2}(i,j) >0.$
Then we have $p_n(i,j) = \gamma p_{n+1}(i,j).$
\item
Assume $A^{n+1}(i,j) >0$. Then we have
\begin{equation*}
p_n(i,j) = \sum_{k=1}^N A(j,k) p_{n+1}(i,k)  
         = \sum_{h=1}^N A(h,i) p_{n+1}(h,j).  
\end{equation*}
\item
Assume $A^{n}(i,j) >0$ and  $A^{n+1}(i,j) >0.$
Then we have 
\begin{equation*}
\gamma p_n(i,j) = \sum_{k=1}^N A(j,k) p_n(i,k)  
         = \sum_{h=1}^N A(h,i) p_n(h,j).  
\end{equation*}
\end{enumerate}
\end{lemma}
\begin{proof}
(i) Since $A^{n+1}(i,j), \, A^{n+2}(i,j) >0,$
we may find
$\nu =(\nu_1,\dots,\nu_n),\, 
 \mu =(\mu_1,\dots,\mu_{n+1})$ 
such that 
$(i,\nu_1,\dots,\nu_n, j) \in B_{n+2}(\bar{X}_A),\,
  (i,\mu_1,\dots,\mu_{n+1}, j) \in B_{n+3}(\bar{X}_A)$.
Consider 
$V_{\nu,\mu}(i,j) = T_i^*T_{\nu_1}^*\cdots T_{\nu_n}^*T_j^*\otimes
S_i S_{\mu_1}\cdots S_{\mu_{n+1}} S_j.
$
It then follows that 
\begin{align*}
p_n(i,j)
=& p_{n,\nu}(i,j) \\
= & \psi(V_{\nu, \mu}(i,j)^*V_{\nu, \mu}(i,j)) \\
= & \psi(V_{\nu, \mu}(i,j) \delta^A_{i\log\gamma}(V_{\nu, \mu}(i,j)^*)) \\
= & \gamma \psi(V_{\nu, \mu}(i,j) V_{\nu, \mu}(i,j)^*) \\
= & \gamma p_{n+1,\mu}(i,j) =\gamma p_{n+1}(i,j).
\end{align*}
(ii) 
Since $A^{n+1}(i,j) >0,$ 
we may find
$\nu =(\nu_1,\dots,\nu_n)$ such that 
$(i,\nu_1,\dots,\nu_n, j) \in B_{n+2}(\bar{X}_A)$.
It then follows that 
\begin{align*}
p_n(i,j)
=& p_{n,\nu}(i,j) \\
= & \psi(T_i T_i^*
\otimes
S_{\nu_1}\cdots S_{\nu_n} S_j
S_j^*S_{\nu_n}^*\cdots S_{\nu_1}^*) \\
= & \sum_{k=1}^N A(j,k)
\psi(T_i T_i^*
\otimes
S_{\nu_1}\cdots S_{\nu_n} S_jS_kS_k^*
S_j^*S_{\nu_n}^*\cdots S_{\nu_1}^*) 
 \\
= & \sum_{k=1}^N A(j,k) p_{n+1,\nu j}(i,k)
=  \sum_{k=1}^N A(j,k) p_{n+1}(i,k).
\end{align*}
We also see that 
\begin{align*}
p_n(i,j)
= & \sum_{h=1}^N A^t(i,h)
\psi(T_i T_h T_i^*T_h^*
\otimes
S_{\nu_1}\cdots S_{\nu_n} S_j S_j^*S_{\nu_n}^*\cdots S_{\nu_1}^*) 
 \\
= & \sum_{h=1}^N A(h,i) p_{n+1,i \nu j}(h,j)
=  \sum_{h=1}^N A(h,i) p_{n+1}(h,j).
\end{align*}
The assertion (iii) follows from (i) and (ii).
\end{proof}
\begin{lemma}\label{lem:6.9}
For $i=1,\dots,N$ and $n \in \Zp$, we have
\begin{enumerate}
\renewcommand{\theenumi}{\roman{enumi}}
\renewcommand{\labelenumi}{\textup{(\theenumi)}}
\item
$
\sum_{j=1}^N A^{n+1}(i,j) p_n(i,j) = \psi(T_i^*T_i \otimes S_i S_i^*)
$
and hence 
$
\sum_{i,j=1}^N A^{n+1}(i,j) p_n(i,j) = 1.
$
\item
$
\sum_{j=1}^N A^{n+1}(j,i) p_n(i,j) = \psi(T_iT_i^* \otimes S_i^* S_i)
$
and hence
$
\sum_{i,j=1}^N A^{n+1}(j,i) p_n(i,j) = 1.
$
\end{enumerate}
\end{lemma}
\begin{proof}
(i)
We have the following identities
\begin{align*}
\psi(T_i^*T_i \otimes S_i S_i^*)
= & \sum_{{\mu_1}=1}^N A(i,\mu_1) 
\psi(T_i^*T_i \otimes S_i S_{\mu_1} S_{\mu_1}^*S_i^*) \\
= & \sum_{j=1}^N \sum_{\mu_1, \dots, \mu_n =1}^N A(i,\mu_1)A(\mu_1,\mu_2)\cdots A(\mu_n,j) 
\psi(T_i^*T_i \otimes S_{i \mu_1\cdots\mu_n j}S_{i \mu_1\cdots\mu_n j}^*) \\
= & \sum_{j=1}^N A^{n+1}(i,j) p_n(i,j).
\end{align*}
We also have  $\sum_{i=1}^N \psi(T_i^*T_i \otimes S_i S_i^*) = \psi(E_A) =1$.

(ii) is similarly shown to (i) .
\end{proof}
We notice that 
$\psi(T_i^*T_i \otimes S_i S_i^*) = \psi(T_i T_i^* \otimes S_i^* S_i)
$
because of the equality $\delta^A_{i\log\gamma}(T_i^* \otimes S_i) =T_i^* \otimes S_i$
and the KMS condition for $\psi$.
Recall that we are assuming the matrix $A$ is aperiodic so that 
there exists $n_0 \in \N$ such that 
$A^n(i,j) >0$ for all $i,j =1,\dots,N$ and $n\ge n_0$.
\begin{lemma}\label{lem:6.10}
 $\gamma = \beta$.
\end{lemma}
\begin{proof}
 Lemma \ref{lem:6.8} together with Lemma \ref{lem:6.9} implies that
the vector $[p_n(i,k)]_{k=1}^N$ is  a nonnegative eigenvector of the matrix $A$
of eigenvalue $\gamma$ for each $n \in \N$ and $i = 1,\dots,N$.
Since $A$ is aperiodic,   $[p_n(i,k)]_{k=1}^N$ is actually  a positive eigenvector 
of eigenvalue $\gamma$.
By Perron--Frobenius theorem, 
$\gamma$ coincides with the Perron-Frobenius eigenvalue $\beta$.
\end{proof}
We have seen that $\gamma$ must be the Perron-Frobenius eigenvalue of the matrix $A$ by Lemma \ref{lem:6.10}.
Its proof does not need the assumption $\gamma>1$
that we had first assumed.   
Now the matrix $A$ is aperiodic and not any permutation
so that its Perron-Frobenius eigenvalue  is always greater than one.
Hence $\gamma(=\beta)$ becomes greater than one without assumption 
 $\gamma>1$.
 
Recall that $[a_j]_{j=1}^N, [b_i]_{i=1}^N$ be the positive eigenvectors 
of $A$ and $A^t$ for the eigenvalue $\beta$ respectively such that 
$\sum_{i=1}^N a_i b_i = 1$. 
We have the following lemma.
\begin{lemma}\label{lem:6.11}
For $n \ge n_0$ and $i,j=1,\dots,N$, we have
\begin{equation}
p_n(i,j) = 
\frac{b_i a_j} {(\sum_{h=1}^N b_h)\cdot(\sum_{k=1}^N a_k)} \sum_{h,k=1}^N p_n(h,k).
\label{eq:6.11.1}
\end{equation}
\end{lemma}
\begin{proof}
We fix $n \ge n_0$.
For a fixed $i=1,\dots,N,$
the vector $[p_n(i,k)]_{k=1}^N$ is  a positive eigenvector 
of the matrix $A$ for the eigenvalue $\beta$.
By the uniqueness of the positive eigenvector of $A$, we may find a positive real number
$c_{n,i}$ such that 
\begin{equation}
p_n(i,j) = c_{n,i}a_j \quad \text{ for } j =1,\dots,N. \label{eq:6.11.0}
\end{equation}
By Lemma \ref{lem:6.8},
we know that the vector
$[\sum_{j=1}^N p_n(i,j)]_{i=1}^N$ 
is a positive eigenvector of the matrix $A^t$ 
for the eigenvalue $\beta$.
Hence the normalized positive eigenvectors
$[\frac{\sum_{j=1}^N p_n(i,j)}{ \sum_{h,k=1}^N p_n(h,k) } ]_{i=1}^N
$
and
$[\frac{b_i}{ \sum_{k=1}^N b_k } ]_{i=1}^N$
coincide, so that 
we have
\begin{equation}
\sum_{j=1}^N p_n(i,j) = b_i \frac{ \sum_{h,k=1}^N p_n(h,k) }{ \sum_{k=1}^N b_k } 
\quad\text{ for } i=1,\dots,N. \label{eq:6.11.2}
\end{equation}
By \eqref{eq:6.11.0} and \eqref{eq:6.11.2}, we have
\begin{equation*} 
c_{n,i} =\frac{\sum_{j=1}^N p_n(i,j)}{\sum_{j=1}^N a_j}
        =\frac{b_i { \sum_{h,k=1}^N p_n(h,k) }  }{( \sum_{k=1}^N b_k) (\sum_{j=1}^N a_j)}
\end{equation*}
so that we know \eqref{eq:6.11.1} by using \eqref{eq:6.11.0} again.
\end{proof}
We thus obtain the following lemma.
\begin{lemma}\label{lem:6.12}
For $n \ge n_0$ and $i,j=1,\dots,N$, we have
\begin{equation*}
p_n(i,j) = 
\frac{1}{\beta^{n+1}}{b_i a_j}.
\end{equation*}
\end{lemma}
\begin{proof}
We fix $n \ge n_0$.
By Lemma \ref{lem:6.11} together with Lemma \ref{lem:6.9}, we have
\begin{equation*}
1 = \sum_{i,j=1}^N A^{n+1}(i,j)p_n(i,j) 
  = \sum_{i,j=1}^N
\frac{A^{n+1}(i,j) b_i a_j} {(\sum_{h=1}^N b_h)\cdot(\sum_{k=1}^N a_k)} \sum_{h,k=1}^N p_n(h,k).
\end{equation*}
As $[a_j]_{j=1}^N$ is a positive eigenvector of $A$ for the eigenvalue $\beta$, we have 
\begin{equation*}
\sum_{i,j=1}^N A^{n+1}(i,j) b_i a_j = \sum_{i=1}^N \beta^{n+1}b_i a_i =  \beta^{n+1}
\end{equation*}
so that the equalities
\begin{equation*}
1 =
\frac{\beta^{n+1}} {(\sum_{h=1}^N b_h)\cdot(\sum_{k=1}^N a_k)} \sum_{h,k=1}^N p_n(h,k)
\end{equation*}
and 
\begin{equation}
\sum_{h,k=1}^N p_n(h,k) = 
\frac{1}{\beta^{n+1}} {(\sum_{i=1}^N b_i)\cdot(\sum_{j=1}^N a_j)}
\label{eq:6.12.3}
\end{equation}
hold. 
By \eqref{eq:6.11.1} and \eqref{eq:6.12.3}, we get the desired equality. 
\end{proof}
Consequently, we know  the following proposition.
\begin{proposition} 
The restriction of a KMS state $\psi$ on $\WRA$ to the commutative $C^*$-subalgebra
$C(\bar{X}_A)$ coincides with the state defined by the Parry measure on $\bar{X}_A$.
\end{proposition}
\begin{proof}
For $n \ge n_0$
and $\xi=(i,\nu_1,\dots,\nu_n,j) \in B_{n+2}(\bar{X}_A)$,
Lemma \ref{lem:6.12} shows  that 
\begin{equation*}
\psi(T_i T_i^* \otimes S_{\nu_1 \cdots\nu_n j}S_{\nu_1 \cdots\nu_n j}^*)
=\frac{1}{\beta^{n+1}}b_i a_j.
\end{equation*}
Let $\mu$ be the Parry measure on $\bar{X}_A$.
Since the Parry measure of the cylinder set 
$U_{[\xi]_m^{m+n+1}} \subset \bar{X}_A, m\in \Z$
for the word $\xi$ is given by 
\begin{equation*}
\mu(U_{[\xi]_m^{m+n+1}}) = \frac{1}{\beta^{n+1}}b_i a_j
\end{equation*}
by the formula \eqref{eq:Parry}.
Let $\chi_{U_{[\xi]_m^{m+n+1}}}$ be the charachteristic function of the cylinder set
$U_{[\xi]_m^{m+n+1}}$.
Since
\begin{equation*}
\psi(T_i T_i^* \otimes S_{\nu_1 \cdots\nu_n j}S_{\nu_1 \cdots\nu_n j}^*)
=\psi(\chi_{U_{[\xi]_m^{m+n+1}}}),
\end{equation*}
we obtain
\begin{equation*}
\mu(U_{[\xi]_m^{m+n+1}})=\psi(\chi_{U_{[\xi]_m^{m+n+1}}}).
\end{equation*}
Any cylinder set on $\bar{X}_A$ is a finite union of cylinder sets of words
having its length greater than $n_0 +1$.
Hence we may conclude that 
the restriction of $\psi$ to the commutative $C^*$-subalgebra
$C(\bar{X}_A)$ of $\WRA$ coincides with the state defined by the Parry measure on $\bar{X}_A$.
\end{proof}
Therefore we reach the following theorem.
\begin{theorem}\label{thm:KMS}
Assume that the matrix $A$ is aperiodic.
A KMS state on $\WRA$ for the action $\delta^A$ at the inverse temperature $\log \gamma$
exists if and only if $\gamma$ is the Perron--Frobenius eigenvalue of $A$.
The admitted KMS state is unique.
The restriction of the admitted KMS state to the subsalgebra $C(\bar{X}_A)$
is the state defined by the Parry measure on $\bar{X}_A$.
\end{theorem}

{\it Acknowledgments:}
This work was supported by JSPS KAKENHI Grant Number 15K04896.

\end{document}